%% file: main.tex
\documentclass[11pt]{article}

\usepackage{hyperref}
\hypersetup{
    colorlinks=True,       
    linkcolor=blue,          
    citecolor=blue,        
    filecolor=magenta,      
    urlcolor=blue           
}

\usepackage[in]{fullpage}
\usepackage[utf8]{inputenc}
\usepackage[T1]{fontenc}
\usepackage{amsfonts,amsmath,amssymb}
\usepackage{amsthm}
\usepackage{graphicx}
\usepackage{etoolbox}
\usepackage{braket}
\usepackage{float}
\usepackage{todonotes}
\usepackage{bm}
\usepackage{colortbl} 
\definecolor{guppiegreen}{rgb}{0.0, 1.0, 0.5}
\definecolor{vermilion}{rgb}{0.89, 0.26, 0.2}
\usepackage{cancel}

\newcommand{\PreserveBackslash}[1]{\let\temp=\\#1\let\\=\temp}
\newcolumntype{C}[1]{>{\PreserveBackslash\centering}p{#1}}

\usepackage[sorting=none,
            style=numeric,
            doi=false,
            isbn=false,
            url=false,
            eprint=false]{biblatex} 
\numberwithin{equation}{section}

\newcommand{\R}{\mathbb{R}}

\newcommand{\cZ}{\mathcal{Z}}
\newcommand{\cH}{\mathcal{H}}

\newcommand{\wT}{\widetilde{T}}
\newcommand{\wH}{\widetilde{H}}
\newcommand{\wE}{\widetilde{E}}

\newcommand{\wl}{\widetilde{\lambda}}

\newcommand{\hH}{\widehat{H}}
\newcommand{\hE}{\widehat{E}}

\newcommand{\bc}{\mathbf{c}}

\newtheorem{theorem}{Theorem}[section]
\newtheorem{corollary}[theorem]{Corollary}
\newtheorem{definition}[theorem]{Definition}
\newtheorem{lemma}[theorem]{Lemma}

\newtheorem{example}[theorem]{Example}

\newtheorem{remark}[theorem]{Remark}

\title{\bf
\rule{\linewidth}{1pt}
Uplifting edges in higher order networks: spectral centralities for non-uniform hypergraphs
\rule{\linewidth}{1pt}
} 

\makeatletter
\renewcommand\@date{{%
  \vspace{-2\baselineskip}%
  \large\centering
  \begin{tabular}{@{}c@{}}
    Gonzalo Contreras-Aso\footnote{Corresponding author: gonzalo.contreras@urjc.es} \textsuperscript{$,\,\sharp$,$\flat$}
  \end{tabular}%
  \quad and \quad
  \begin{tabular}{@{}c@{}}
    Cristian Pérez-Corral\textsuperscript{\textsection}
  \end{tabular}%
  \quad and\quad
  \begin{tabular}{@{}c@{}}
    Miguel Romance\textsuperscript{$\sharp$,$\flat$}
  \end{tabular}

  \bigskip

{\it \normalsize
  \textsuperscript{$\sharp$}Departamento de Matem\'atica Aplicada, Ciencia e Ingenier\'{\i}a de los Materiales y Tecnolog\'{\i}a Electr\'onica, Universidad Rey Juan Carlos, 28933 M\'ostoles (Madrid), Spain\par

  \textsuperscript{$\flat$}Valencian Research Institute for Artificial Intelligence, Universitat Politècnica de València, 46022 València (Valencia), Spain 
  
  \textsuperscript{\textsection}Laboratorio de Computaci\'on Matem\'atica en Redes Complejas y sus Aplicaciones, Universidad Rey Juan Carlos, 28933 M\'ostoles (Madrid), Spain

}

  \bigskip

  \today
}}
\makeatother

\begin{document}

\maketitle

\begin{abstract}
    Spectral analysis of networks states that many structural properties of graphs, such as centrality of their nodes, are given in terms of their adjacency matrices. The natural extension of such spectral analysis to higher order networks is strongly limited by the fact that a given hypergraph could have several different adjacency hypermatrices, hence the results obtained so far are mainly restricted to the class of uniform hypergraphs, which leaves many real systems unattended. A new method for analysing non-linear eigenvector-like centrality measures of non-uniform hypergraphs is presented in this paper that could be useful for studying properties of $\cH$-eigenvectors and $\cZ$-eigenvectors in the non-uniform case. In order to do so,  a new operation - the \textit{uplift} - is introduced, incorporating auxiliary nodes in the hypergraph to allow for a uniform-like analysis and we furthermore use it to classify a whole family of hypergraphs with unique Perron-like $\cZ$-eigenvectors. We supplement the theoretical analysis with several examples and numerical simulations on synthetic and real datasets.
\end{abstract}

\newpage

\section{Introduction}\label{sec:introduction}
\input{introduction}

\section{Preliminaries}\label{sec:preliminaries}
\input{preliminaries}

\section{The uplift}\label{sec:uplift}
\input{analytical}


\section{Uplift + $\cH$-eigenvectors: spectral centrality in non-uniform hypergraphs}\label{subsec:upliftHEC}
\input{HEC-uplift}

\subsection{Numerical comparisons}
\input{HEC-numerical}


\section{Uplift + $\cZ$-eigenvectors: an uniqueness sufficient condition}\label{subsec:upliftZeigs}
\input{ZEC-uplift}



\newpage

\section{Conclusions}\label{sec:conclusions}
\input{conclusions}

\subsection{Code and data availability}

The datasets used in the numerical simulations throughout this article, as well as the code used to analyze them, can be found in the repository \par 
\url{https://github.com/LaComarca-Lab/non-uniform-hypergraphs}.

\subsection{Acknowledgements}
This work has been partially supported by INCIBE/URJC Agreement M3386/2024/0031/001 within the framework of the Recovery, Transformation and Resilience Plan funds of the European Union (Next Generation EU) and by projects M2978 and M3033 (URJC Grants). G. C-A acknowledges funding from the URJC fellowship PREDOC-21-026-2164.

\printbibliography

\end{document}

%% file: introduction.tex
The last decade has a rise in the multidisciplinary field of ``complexity science'', partly due to the advances in computation, but also due to important milestones in theory. Within this far-reaching field, one of the most active and important areas is that of complex networks: the study of systems reduced to individuals and their interactions \cite{boccaletti2006structure}. This area was originated in the mathematics of graph theory, but it was also supplemented by ideas and concepts from statistical physics, biology, computation science, to mention a few, yielding an ever increasing amount of results in these and other areas, such as sociology, dynamical systems and even multilinear algebra. 

Theoretical advances in complex network theory have shed light on its own limitations; in particular these last years a lot of effort has been poured in extending the analytical, conceptual and numerical tools already available for graphs to the realm of hypergraphs. Hypergraphs are natural extensions to graphs, where interactions are not restricted to be of a pairwise nature, i.e. one can have interactions between two, three, four or even higher individuals. These are sometimes referred to as higher order interactions, or hyperedges for short. The quintessential example of these kinds of systems is that of co-authorship networks: there, nodes represent authors and scientific articles are the interactions between them (which clearly are not restricted to papers written by pairs of authors). 

Although there is already a plethora of results brought from classic complex network theory \cite{battiston2020networks, boccaletti2023structure}, there is an even higher amount of results which have yet to find their way into hypergraphs, to the point that it is a current area of research (see e.g. \cite{liu2023eigenvector}). The reason is pretty simple: considering higher order interactions usually complicates analytical calculations to the point where certain approximations or assumptions need to be established in order to make any progress at all. An illustrative example is that of obtaining spectral properties, such as spectral centrality measures, which provide a way to quantify the importance of nodes beyond their degrees, taking into account the relevance of their neighbors, and it is the basis of some real-world applications such as PageRank \cite{page1998pagerank}. In in graphs the study of these properties translates to studying matrices (adjacency, Laplacian, etc) whereas in hypergraphs translates to studying tensors, and while the former has been extensively studied, the latter has not. 

The previous example relates to the subject of this paper. In order to generalize spectral centralities of graphs to the case of hypergraphs, Benson \cite{benson2019three} makes use of the recently developed theory of $\cH$-eigenvectors and $\cZ$-eigenvectors of tensors \cite{qi2017tensor}, to formulate a series of centrality measures. This formulation does apply, however, to uniform hypergraphs; those where the hyperedges are all of the same order, which strongly limits the application of the obtained results to many real systems that only could be modelled by using hypergraphs with hyperedges of different sizes. The main goal of this paper is presenting a method that allows considering non-linear eigenvector-like centrality measures of non-uniform hypergraphs and therefore extends the results obtained by Benson to uniform higher order networks \cite{benson2019three}. 

The technique we present is based on incorporating auxiliary nodes in a non-uniform hypergraph in order to enable a uniform-like analysis, such that we can use all the tools and results obtained for uniform hypergraphs, but now for a general one. This uniformization process, named \textit{uplift} of an hypergraph, is somehow the dual of the projection process that transforms hyperedges among $k$ nodes into hyperedges of $2\le j<k$ nodes and it has sound mathematical properties that make it very useful for analyzing several structural properties of hypergraphs and hypermatrices, as it will be shown along this paper. 

It is important to point out that, although we will only focus on the analytical tools proposed and analyzed in this manuscript within the subject of network theory and centrality measures, the methods can be leveraged in all other areas where tensor eigenproblems have made an appearance, such as biology \cite{bini2011solution}, medical imaging \cite{qi2008eigenvalues}, quantum entanglement \cite{hu2016computing}, data mining \cite{benson2015tensor}, or higher order Markov chains \cite{ng2009finding}.

This article is structured as follows. In Section~\ref{sec:preliminaries} we provide the mathematical notions required throughout, as well as present a summary of the state of the art. In Section~\ref{sec:uplift}, we define the \textit{uplift} operation, which allows us to uniformize any hypergraph in a way that appropriately generalizes the uniform $\cH$-eigenvector centrality, something which is then discussed in Section~\ref{subsec:upliftHEC}. We show why it fails to generalize the $\cZ$-eigenvector centrality, although we rescue it in Section~\ref{subsec:upliftZeigs} to solve a problem in multilinear algebra, which in turn feeds back into the problem of $\cZ$-eigenvector centrality of certain hypergraphs. Several examples and numerical computations on a variety of synthetic and real higher order networks are included along the paper in order to illustrate the analytical results presented and compare them with other proposals and methods.


%% file: preliminaries.tex

We will start by giving a brief overview of the main concepts and notation which will be needed in what follows, first regarding the eigenvector centrality in standard graphs and then introducing useful hypergraph notions.

\paragraph{The eigenvector centrality in graphs:}
This work is framed in the context of centralities measures in networks \cite{boccaletti2006structure}. A centrality measure is simply a way to provide a notion of importance to the nodes within a network, based on a criteria (heuristics) regarding what constitutes importance. This criteria is subjective, and must be decided with an application in mind. The simplest criteria is counting the number of neighbors of each node, this is the so-called degree centrality. However, for many real-world applications (see \cite{vigna2019spectral} and the references therein) we are interested in also considering the importance of each neighbor in the computation of a node's importance. That is the basis for all spectral centralities.

The simplest spectral centrality in standard, pairwise interaction networks is the so-called Eigenvector Centrality. The heuristics behind it is the statement that in a graph $G=(V,E)$, a node's importance is proportional to the importance of its neighbors \cite{estrada2012structure}. Mathematically,
\begin{equation}\label{eq:pairwise-eigcent}
    c_i \propto \sum_{j\rightarrow i} c_j \Rightarrow \lambda \bc = A^T \bc,
\end{equation}
where $A$ is the adjacency matrix of the graph $G$. By the Perron-Frobenius Theorem, if the graph $G$ is connected then $A$ is irreducible, and therefore the existence and uniqueness of a positive eigenvector $\bc$ associated to the spectral radius $\rho(A)$ is guaranteed \cite{meyer2001matrix}. This theorem has provided support for a plethora of spectral centralities in standard complex network theory, out of which the Eigenvector Centrality is the paradigmatic example. 


\paragraph{Notions from algebraic hypergraph theory:}
A hypergraph $H = (V, E)$ consists of a set of nodes $V=\{i_1, ..., i_n\}$ and a set of edges $E$. Each hyperedge consists of yet another set (or multiset, as will be discussed in \ref{sec:uplift}) of nodes belonging to $V$. The size of a hyperedge is the number of elements within it. If the hypergraph is weighted, then there exists a function $w: E \rightarrow \mathbb{R}$, and $w(e)$ is the weight of edge $e$. 

We say that a hypergraph is $m$-uniform if all of its hyperedges are of the same size $m$. Note that this subclass of hypergraphs is uncommon in real networks: if we consider the quintessential example of hypergraphs, which is the network of collaboration between scientists, the number of authors (nodes) in each hyperedge (paper) might not always the same. Note also that the case of 2-uniform hypergraphs coincides precisely with networks of pairwise interactions. 

Working with $m$-uniform tensors, if possible, is preferable as there are several analytical tools available for them. Namely, one can define the ``hypergraph adjacency tensor'' $T\in\R^{[m,n]}$, whose components are\footnote{The name ``tensor'' is usually reserved for mathematical objects invariant under coordinate transformations. In our case we are instead referring to multidimensional arrays (or hypermatrices), which we refer to as tensors for the sake of conciseness.}
\begin{equation}
    T_{i_1 ... i_m} = \begin{cases}
    w(\{i_1 ... i_m\}) & \text{if } \{i_1,...,i_m\} \in E\\
    0 & \text{otherwise.}
    \end{cases}, \quad 1 \leq i_1,...,i_m \leq n,
\end{equation}
where $m$ is the ``order'' of the tensor (equivalently, the size of its hyperedges). In the context of undirected hypergraphs, the tensors which we will be discussing will always be symmetrical, meaning that $T_{i_1 ... i_k}=T_{\sigma\{i_1,...,i_m\}}$, for all $\sigma\in \mathfrak{S}_m$, where $\mathfrak{S}_m$ denotes the permutation group of $m$ indices.

The key ingredient connecting the Perron-Frobenius Theorem to the eigenvector centrality of graphs was the relation between the irreducibility of the adjacency matrix and the strong connectivity of the graph. In hypergraphs, generalizations of the Perron-Frobenius Theorem have been put forward for different tensor eigenproblems \cite{qi2017tensor}, which do require a similar connection to be made between irreducibility of the adjacency tensor and strong connectivity of the corresponing hypergraph.

Similarly to the matricial case, we can distinguish between reducible and irreducible tensors, but with a bit of a twist.
\begin{definition}[Irreducible tensor \cite{qi2017tensor}] \label{def:irreducible}
An order-$m$, dimension-$n$ tensor T is reducible if there is a nonempty proper index subset $J \subset [n]$ such that
\begin{equation}
    T_{i_1 ... i_m} = 0,\, \forall i_1 \in J,\, \forall i_2, ... , i_m \notin J.
\end{equation}
If T is not reducible, then it is irreducible.
\end{definition}

Here we used the notation $[n]=\{1,...,n\}$. A number of results have been proven relating connectedness properties of hypergraphs to the irreducibility of this tensor \cite{pearson2014spectral, qi2017tensor}. However, unlike the pairwise case, the intuitive notion of connectedness in a hypergraph does not directly translate to irreducibility of the associated tensor (See Example 2.7 of \cite{chang2013variational}). As it happens, tensor irreducibility is too strong a constraint to fully describe general hypergraphs. 

Instead, strongly connected hypergraphs are described in terms of \textit{weakly} irreducible tensors. 
\begin{definition}[Weakly irreducible tensor \cite{qi2017tensor}] \label{def:weakly-irreducible}
Let $M=(m_{ij})$ be a $n \times n$ non-negative matrix defined by
\begin{equation}
m_{ij} = \sum_{j_3 , ..., j_m = 1}^N |T_{i j j_3 \dots j_m} |.
\end{equation}
Then $T$ is weakly irreducible if and only if $M$ is irreducible.
\end{definition}

This definition is equivalent to the intuitive notion of connectedness, as the graph whose adjacency matrix is $M$ will have an edge between nodes $i,j$ if there is at least a hyperedge containing them. 

A hypergraph is said to be strongly connected if its adjacency tensor is irreducible in the previous sense. Luckily, most of the existence and uniqueness results which will be relevant for us have also been proven for these types of tensors \cite{qi2017tensor}

Before moving on to the next subsection, let us define an ubiquitous operation involving a tensor $T\in\R^{[m,n]}$ and a vector $\bc \in \R^n$, whose contraction produces yet another vector, sometimes called \textit{tensor apply} \cite{benson2019computing} or TTSV1 \cite{aksoy2024scalable}:
\begin{equation}
    \bm{x} = T \bc^{m-1} \Longleftrightarrow  x_{i_1} = \sum_{i_2...,i_m=1}^n T_{i_1 i_2 ... i_m} \bc_{i_2} ... \bc_{i_m}.
\end{equation}

\subsection{Hypergraph spectral centralities: State of the Art}\label{subsec:art}

For simplicity we will now restrict ourselves to $3$-uniform hypergraphs, although the generalization to $k$-uniform hypergraphs is straightforward. The most straightforward generalization of the pairwise Eigenvector Centrality consists of defining functions $f:V\rightarrow \mathbb{R}$ and $g:V\times V \rightarrow \mathbb{R}$, and imposing the equation
\begin{equation}\label{eq:3uniform-fg}
    f(c_i) = \frac{1}{\lambda} \sum_{\{i,j,k\} \in E} g(c_j,c_k).
\end{equation}

The question then would be determining whether $\mathbf{c}=(c_1,...,c_n)^T$ exists and is unique, and if so how could one calculate it. 

So far in the literature three choices of $f,g$ have been considered \cite{benson2019three}, due to their simplicity, sensibility and the existence of Perron-Frobenius-like associated theorems.
\begin{itemize}

    \item \textbf{Clique motif Eigenvector Centrality (CEC)}: In this case $f(c_i) = c_i$ and $g(c_j,c_k) = c_j + c_k$. This choice is the simplest one, for it leads to a standard eigenvector equation, unlike in the next two cases.
    \begin{equation}\label{eq:3uniform-cec}
        c_i = \frac{1}{\lambda} \sum_{\{i,j,k\} \in E} (c_j + c_k).
    \end{equation}
    This is tantamount to considering the standard (pairwise) Eigenvector Centrality of the motif adjacency matrix $W$ of the hypergraph \cite{benson2019three}.  

    The main drawback of this approach is that it hides the higher order nature of the hypergraph, reducing the problem of computing its centrality scores to that of a standard graph with a modified adjacency matrix. 
    
    \item \textbf{$\cZ$-Eigenvector Centrality (ZEC)}: In this case $f(c_i) = c_i$ and $g(c_j,c_k) = c_j c_k$. This choice fully incorporates the higher order nature of the hypergraph by means of a non-linear $g$. 
    \begin{equation}\label{eq:3uniform-zec}
        c_i = \frac{1}{\lambda} \sum_{\{i,j,k\} \in E} c_j \,c_k = \frac{1}{\lambda} \sum_{j,k=1}^n T_{ijk}\, c_j \, c_k\Rightarrow \lambda \bc = 2 T \bc^2.
    \end{equation}

    \begin{remark}
    Note that, if the eigenpair $(\lambda, \bc)$ is a solution to this equation, then for any $\alpha\in \R$ the eigenpair $(\alpha\lambda, \alpha\bc)$ is also a solution. This is problematic, as it means there are infinite eigenvalues. To deal with this, the unitarity condition $\bc^T\bc=1$ is also imposed\footnote{Vectors satisfying $\bc^T \bc=||\bc||_2=1$ (Euclidean norm) are sometimes referred to as $\cZ_2$-eigenvectors in the literature \cite{chang2013uniqueness, qi2017tensor}, as opposed to $\cZ_1$-eigenvectors satisfying $|\bc|_1=1$ ($\ell_1$-norm).}, although it makes the $\cZ$-eigenvector $\bc$ not re-scalable.
    \end{remark}
    
    This equation was proven to always have a positive $\bc>0$ solution, provided the underlying hypergraph is strongly connected \cite{chang2008perron}. Moreover, this definition of eigenvectors, unlike the next one, is invariant under orthogonal transformations \cite{qi2017tensor}, making it physically relevant.
    
    \begin{remark}
    This spectral problem features a different behavior depending on whether the order of the tensor $T$ is even or odd. In the former case, for every eigenpair $(\lambda, \bc)$, the eigenpair $(\lambda, -\bc)$ is also a solution. In the latter case, for every eigenpair $(\lambda, \bc)$, the eigenpair $(-\lambda, -\bc)$ is also a solution. Therefore, in that case the spectral radius does not correspond to a unique eigenvalue, however for centrality purposes we are only interested in the ranking, and we can thus choose to keep just the positive solution.
    \end{remark}

    However, this measure has two important flaws: firstly, the solution is not unique (even after imposing unitarity, see Example 2.7 of \cite{chang2013variational}), and secondly, all computational methods known to converge to a solution are computationally expensive \cite{qi2017tensor}.  
    
    \item \textbf{$\cH$-Eigenvector Centrality (HEC)}: In this case $f(c_i) = c_i^2$ and $g(c_j,c_k) = c_j c_k$. In this case we consider the same non-linear function $g$, but consider a function $f$ which ``dimensionally match'' the right hand side (supposing centrality is measured in some ``unit''). 
    \begin{equation}\label{eq:3uniform-hec}
        c_i^2 = \frac{1}{\lambda} \sum_{\{i,j,k\} \in E} c_j \, c_k = \frac{1}{\lambda} \sum_{j,k=1}^n T_{ijk} \, c_j \,c_k\Rightarrow  \lambda \bc^{[2]} = 2 T \bc^2,
    \end{equation}
    where $\bc^{[2]}$ refers to the Hadamard (componentwise) square of the vector $\bc$. 

    This equation not only is guaranteed to have a positive $\bc>0$ solution if the hypergraph is strongly connected, but said solution is also unique (up to scaling) \cite{chang2008perron}.  
\end{itemize}

We will not further discuss the properties, advantages and disadvantages of each of these methods. The interested reader is referred to \cite{benson2019three} for an in-depth analysis of them in the context of hypergraphs, and to \cite{qi2017tensor} for a general discussion of their mathematical properties, special cases and results.

Up to this point we have limited the discussion to $m$-uniform hypergraphs. However, most real hypergraphs are not uniform, having edges of a variety of sizes. For instance, one of the prime examples of real hypergraphs is that of collaboration networks, where nodes represent researchers and hyperedges represent papers where their authors have collaborated. In this simple example it is clear that the corresponding hypergraph will contain pairwise interactions (papers with two authors), triple interactions (papers with three authors), and so on. 

\subsubsection{Vectorial characterizations}

A first idea on how to characterize the centrality of nodes in a non-uniform hypergraph is introducing the notion of a \textit{vectorial centrality score}. For instance, in the HEC case one could compute the centrality vector $\textbf{c}^{(m)}\in \mathbb{R}^n$ associated to each order $m$ sub-hypergraphs (provided they are all strongly connected), and then assign to each $i$ node a vector $\textbf{v}^i\in \mathbb{R}^{m-1}$ whose $m$-th component corresponds to its centrality score at order $m$, i.e. $v_{m}^i=c^{(m)}_i$. 

The main drawback of the above approach is the fact that the hypergraph will likely not be strongly connected at each order (see, for instance, \ref{fig:uplift}). Not only that, but there interplay between different orders is completely absent. An attempt to palliate both problems was put forward in \cite{kovalenko2022vector}, where they resort to the line graph of a hypergraph (a structure which is proven to be strongly connected if the original hypergraph is as well) to translate the problem to that of hyperedge centrality scores, which can be tackled using standard, pairwise graph theory. These edge centralities are then ``shared'' among the nodes participating in each of them, at each level $k$, conforming a vectorial centrality score.

\subsubsection{Embracing non-uniformity} 

Apart from the vectorial characterizations of the hypergraph centrality, it seems rather natural to wonder whether one can extend the notions of CEC, ZEC and HEC to these non-uniform cases.

Following the traditional heuristics of the Eigenvector Centrality, the most general scenario would then be an equation of the form
\begin{equation}\label{eq:nonuniform-fgg}
    \lambda f(c_{i_1}) = \alpha_2 \sum_{\{i_1,i_2\} \in E} g_2(c_{i_2}) + \alpha_3 \sum_{\{i_1,i_2,i_3\} \in E} g_3(c_{i_2},c_{i_3}) + \dots + \alpha_m \sum_{ \{i_1, \dots ,i_m\} \in E} g_m(c_{i_2}, c_{i_3}, \dots, c_{i_m}).
\end{equation}

If some of the functions $f,g_2,g_3, \dots , g_m$ are non-linear, as in the HEC and ZEC cases, this equation is not known to have a solution in general. Not only that, but taking insight from the $m$-uniform case, it is clear that the choices of said functions will drastically change the existence and uniqueness properties of the solution. Our best shot at making progress in the non-uniform, non-linear case is then returning to the ZEC and HEC cases, and see if there is any way to introduce the non-uniformity there.

%% file: analytical.tex
We start with the uniformization of the hypergraph from the bottom up. For that, consider a hypergraph $H=(V,E)$ whose maximum hyperedge size is $M$, and a size $m\geq M$, possibly with multiset hyperedges, i.e. hyperedges where the same node is contained more than once (hypergraphs with such particularity are sometimes referred to as ``multihypergraphs'' \cite{qi2017tensor}). We can turn every hyperedge of size lower than $m$ into that size by adding an auxiliary node\footnote{ This notion of adding extra nodes is already present in other hypergraph-related works, although with completely different purposes: In \cite{zhen2022community}) they call it ``augmentation'' and use it for community detection, in \cite{ouvrard2018adjacency} they call it ``inflation'' and use it for hypergraph polynomials. In either case they use a simpler, unweighted version, which only encodes adjacency and not strength of it, in contrast with our proposal.},  which we name ``$\star$'', possibly multiple times within the same hyperedge.

More precisely, we have the following definition.
\begin{definition}[Uplifted hypergraph at order $m$]
Let $H = (V, E)$ be a hypergraph whose maximum hyperedge size is $M$ and let $m \geq M$. We define the uplifted hypergraph at order $m$ as
\begin{equation}
    \widetilde{H}=(\widetilde{V}, \wE), \quad \text{where} \quad \widetilde{V}=V\cup \{\star\} \quad \text{and} \quad \widetilde{E}=\left\{ e \cup \left(\bigcup_{l=0}^{m-|e|} \{\star\} \right),  \, e\in E \right\}.
\end{equation}
\end{definition}

\begin{remark}
Note that, from a set-theoretic point of view, uplifted hyperedges are multiset objects, i.e. they may contain the auxiliary node $\star$ more than once.
\end{remark}

To exemplify this concept, Figure~\ref{fig:uplift} illustrates the uplifting procedure with a simple case (a hypergraph with two 2-hyperedges and one 3-hyperedge uplifted to order 3 with an auxiliary node). 

\begin{figure}[!ht]
    \centering
    \includegraphics[width=0.7\textwidth]{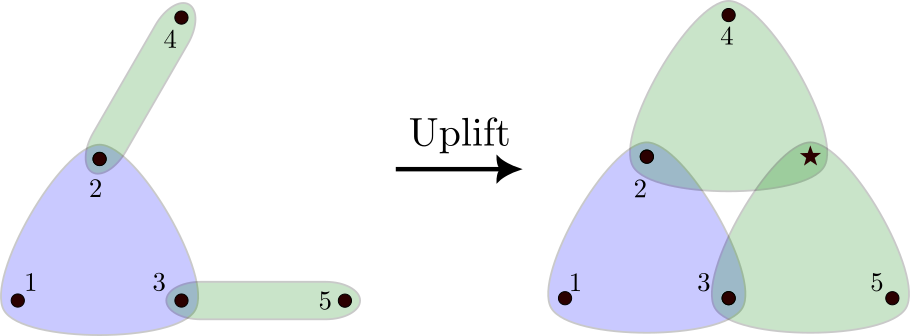}
    \caption{Uplift of the hypergraph $H=(\{1,2,3,4,5\},\{\{1,2,3\},\{2,4\},\{3,5\}\})$ to $\widetilde{H}$.}
    \label{fig:uplift}
\end{figure}

The next step is constructing its associated tensor $\wT$, in particular its components $\wT_{i_1 i_2 ... i_m}$, which will be used in the centrality calculation. In order to do so, one can, naïvely, identify $\star$ as the node $n+1$ and start filling in the entries of the tensor. There's a caveat, though: as we are considering undirected hypergraphs, the tensors at each order are considered symmetrized. Adding the extra node would provide more permutations to each hyperedge than those originally present. We can avoid this double counting by adding suitable combinatorial factors to the hyperedges which have been uplifted.

Taking this into account, we can define the uplifted tensor.
\begin{definition}[Uplifted tensor of a hypergraph] \label{def:uplift-tensor}
Let $H=(V,E)$ be an unweighted, uniform hypergraph and let $\wH=(\widetilde{V}, \wE)$ be its uplift to order $m$. The components of the uplifted tensor $\wT_{i_1 i_2 ... i_m}$ associated to $\wH$ are defined as 
\begin{equation} \label{eq:star-tensor}
    \widetilde{T}_{i_1 i_2 ... i_m} =
    \begin{cases}
    1 & \text{if} \quad \{i_1, i_2, ..., i_m\} \in E, \\
    \dfrac{m^\star (m-m^\star)!}{m!} & \text{if} \quad \{i_1, i_2, ..., i_m\} \notin E \, \wedge \, \{i_1,i_2, ... ,i_m\} \in \wE,\\
    0 & \text{otherwise,}
    \end{cases}
\end{equation}
where $i_s\in \widetilde{V}, \, s = 1,...,m$ and with $m^\star$ being the number of times the auxiliary node $\star$ was added to the original hyperedge $e = \{j_1,...,j_{m-m^*}\}\subset \{i_1,...,i_m\}$ during the uplifting procedure.
\end{definition}

Note that this tensor is weighted (although still non-negative) by construction. An analogous construction could be consider for weighted hypergraph, although we have omitted it for the sake of clarity.

Continuing with the example hypergraph considered in Figure~\ref{fig:uplift}, we would have 
\begin{equation}
    T_{\sigma(123)}=1, T_{\sigma(24)}=1, T_{\sigma(35)}=1 \xrightarrow[]{\rm Uplift} T_{\sigma(123)} = 1, T_{\sigma(24\star)} = 1/3, T_{\sigma(35\star)} = 1/3, 
\end{equation}
where $\sigma(ij)\in\mathfrak{S}_2,\, \sigma(ijk)\in\mathfrak{S}_3$ denote any possible permutation of the indices. This is sensible because, for instance, previously there were two components describing the hyperedge $\{2,4\}$, and now there are 6 describing $\{2,4,\star\}$ but with diminished weight.

We want to stress the importance of this weighting choice, which sets the uplift apart from \cite{ouvrard2018adjacency,zhen2022community}: without it, we would not be able to claim that it appropriately generalizes the uniform HEC, as it would introduce spurious connectivity strengths. This ensures that the total adjacency relations are preserved. In the two cited works the actual strength of the connections is not considered, only the connectivity structure.

If the original hypergraph is strongly connected, the addition of a node to some hyperedges will not hinder the connectivity of the resulting hypergraph. Therefore, in that case it is simple to conclude the strong connectedness of the resulting (uplifted) hypergraph.

\begin{lemma}[Strong connectedness of the uplifted hypergraph] \label{lem:uplifted-connected}
    Let $H=(V,E)$ be a strongly connected hypergraph whose maximum hyperedge size is $M$ and let $m>M$. Then, the uplifted hypergraph $\widetilde{H}$ is strongly connected.
\end{lemma}

It is worth noting that, if the original hypergraph is disconnected, then the uplift may connect it. This enhanced connectivity is an artifact of the operation, and even though one will get a unique, well-defined centrality for that hypergraph, it might be of interest to carefully check whether that is preferable in the application/system in mind.

    


\paragraph{An uplifted nuance.}

Given that the uplift produces uniform hypergraphs, one could now think of apply either HEC or ZEC to the uplifted hypergraph, in order to obtain the centrality score of each node $c_i$. It turns out that there is a nuance in the uplifting procedure which impedes its usage in the ZEC case, which is not an issue in the HEC case, as we will now see.

Take, for example, a hypergraph $H$ with only size 2 (pairwise) and 3 (triple) interactions. Suppose we uplift it adding an auxiliary node $\star$ once inside each of the pairwise edges. The aim of this procedure boils down to enabling the following rewriting of the ``tensor apply'' (present in both HEC \ref{eq:3uniform-hec} and ZEC \ref{eq:3uniform-zec}) operation, grouping all the interactions together
\begin{equation}
    \sum_{j=1}^n a^{(2)}_{ij} c_j + \sum_{j,k=1}^n T^{(3)}_{ijk} c_j \, c_k \rightarrow \sum_{j,k=1}^n \wT_{ijk} c_j \, c_k + \sum_{k=1}^n \wT_{i\star k} c_\star \, c_k + \sum_{j=1}^n \wT_{ij\star} c_j \, c_\star = \wT \bc \, \bc.
\end{equation}
However, this involves a sum (the pairwise one) where the centrality of $c_\star$ would be involved. It would thus seem that there is a flaw in using the uplift to obtain centralities of the original hypergraph: they would depend on the centrality of the spurious node $\star$. Luckily, in the HEC centrality this is not the case, as we will show that this is an artifact of the procedure which we can get rid of.  

To see this, consider applying HEC to an uplifted hypergraph $\widetilde{H}$, obtaining a centrality vector $\bc = (c_1,...,c_n, c_\star)^T$. As discussed before, if $\wH$ is strongly connected, $\bc$ is positive and unique up to scaling \cite{chang2008perron}. Therefore, one can always rescale it such that $c_\star=1$. This choice solves the previously mentioned apparent contradiction, as the sums before and after the uplift would now coincide.

Moreover, given that the centrality scores we care about are just those of the ``real'' nodes, one would then just keep the centrality components associated to them, which can then be rescaled again at will (for instance, in order to normalize them). Hence, the initial scaling to achieve $c_\star$ was just a formal consistency check, but it can conveniently be ignored once we have computed the HEC solution.

\paragraph{The ZEC problem.}

Notice that the reason why we could ignore the aforementioned issue in the HEC case is the fact that if $\bc$ is an $\cH$-eigenvector, then $\bc'\propto \bc$ is still a $\cH$-eigenvector. This is not the case for $\cZ$-eigenvectors: they can't be rescaled and still solve the $\cZ$-eigenproblem defined in Subsection \ref{subsec:art} (recall that they are subject to a normalization constraint \cite{qi2017tensor}). 

It would seem that there is no use for the combination of uplift and $\cZ$-eigenvectors. That turns out not to be the case, if we uplift an already $2$-uniform hypergraph to a $(2+m)$-uniform hypergraph, as we will see. From the point of view of computing importance scores this is unnecessary (the ZEC could already be computed in the original hypergraph), but we will see that it plays an important role in the characterization of Perron-like  $\cZ$-eigenvectors for certain types of hypergraphs. This result will, in turn, feed back into the ZEC centrality quite naturally.

Therefore, from now on we will separate the discussion in two parts: on the one hand, if one starts with a non-uniform hypergraph, its uplift can be used to compute HEC-like centralities. On the other hand, if one starts with a uniform hypergraph, its uplift can shed light on properties of certain $\cZ$-eigenproblems.

%% file: HEC-uplift.tex
We will now particularize what we have been discussing to the case of $\cH$-eigenvectors,. As we mentioned, the main interest of this uplift is the extension of the HEC centrality measure to the case of non-uniform hypergraphs. Given what we know so far, we can already do so.




\begin{definition}[$m$-UHEC] \label{def:m-UHEC}
    Let $H=(V,E)$ be a strongly connected hypergraph whose maximum hyperedge size is $M$ and let $m \geq M$. The $m$-Uplifted $\cH$-Eigenvector Centrality ($m$-UHEC) of the hypergraph $H$ consists of the $n=|V|$ components associated to nodes in $V$ of the HEC of the uplift of $H$ to order $m$.
\end{definition}

For the sake of conciseness and to avoid cluttering the notation, we will from now on refer to the $m$-UHEC as just the UHEC, with the order being clear by the context, or specified otherwise.

Note that if $H$ is already $M$-uniform and $m=M$, then the UHEC and standard HEC vectors coincide. It is straightforward to see that this measure is well-defined in the sense that the UHEC vector is positive and unique (up to scaling), as was the HEC measure.
\begin{theorem}[Existence and uniqueness of the UHEC]
    Under the assumptions of Definition \ref{def:m-UHEC}, the UHEC vector exists and it is unique.
\end{theorem}

\begin{proof}
    Lemma \ref{lem:uplifted-connected} guarantees the strong connectedness of the uplifted hypergraph, and the Perron-Frobenius theorem for strongly connected hypergraphs \cite{chang2008perron} guarantees the existence and uniqueness of its HEC. 
\end{proof}

\subsection{The pairwise case}

The uplift procedure is not restricted to higher-order networks: one can also apply it to pairwise interaction networks. While in real applications there is no obvious reason why one would prefer it to other, well-established, spectral centrality measures, for us it will be interesting to discuss it as a means of comparison with them: if the centrality outcome after the uplift into a hypergraph was considerably different from the centrality outcome of the pairwise eigenvector centrality, that would have signaled a flaw in our approach.

Consider a pairwise interaction graph $G=(V, E)$ with adjacency matrix $A=(a_{ij})$. Uplift it to $\widetilde{H}^3=(\widetilde{N},\widetilde{E})$. Its $3$-uniform tensor can be decomposed as
\begin{equation}
    \widetilde{T}_{i_1 i_2 i_3} = 
    \begin{cases}
        a_{i_1 i_2}/3 & \text{if } i_3=\star \\
        a_{i_2 i_3}/3 & \text{if } i_1=\star \\
        a_{i_1 i_3}/3 & \text{if } i_2=\star \\
        0 & \text{otherwise}.
    \end{cases}
\end{equation}

With this decomposition, one can rewrite the HEC equation \ref{eq:3uniform-hec} as
\begin{align}
    \lambda c_i^2 &= \cancel{\sum_{j,k=1}^n T_{ijk} c_j \, c_k} + \sum_{j=1}^n T_{ij\star} c_j \, c_\star + \cancel{T_{i\star\star} c_\star \, c_\star} = \frac{1}{3}\sum_{j=1}^n a_{ij} c_j \, c_\star, \\
    \lambda c_\star^2 &= \sum_{j,k=1}^n T_{\star jk} c_j \, c_k + \cancel{\sum_{j=1}^n T_{\star j\star} c_j \, c_\star} + \cancel{T_{\star\star\star} c_\star \, c_\star} = \frac{1}{3}\sum_{j,k=1}^n a_{jk} c_j \, c_k .
\end{align}

The second of these equations is not actually relevant: it just ties the value of the auxiliary node based on the scores of the rest. Requiring now the centrality of the auxiliary node to be $c_\star=1$ we end up with 
\begin{equation} \label{eq:HEC-pairwise}
    \lambda c_i^2 = \frac{1}{3}\sum_{j=1}^n a_{ij} c_j,
\end{equation}
which resembles, up to the $c_i^2$, a weighted, undirected version of the Eigenvector Centrality equation \eqref{eq:pairwise-eigcent}. It is therefore natural to compare the centralities obtained through this uplifted measure to those of the standard (pairwise) eigenvector centrality.

We expect to obtain a similar ranking (in the sense of ordering of nodes by importance), although with a lower spread in the actual centrality scores. This is because, loosely speaking, the uplift compresses the centrality scores: the auxiliary node ties every node together, homogenizing the centrality. The most notable thing, is however, the fact that this homogenization can change the actual ranking between the nodes, as can be seen in Figure~\ref{fig:pairwise_graphs}. 

\begin{figure}[!ht]
    \centering
    \includegraphics[width=1\textwidth]{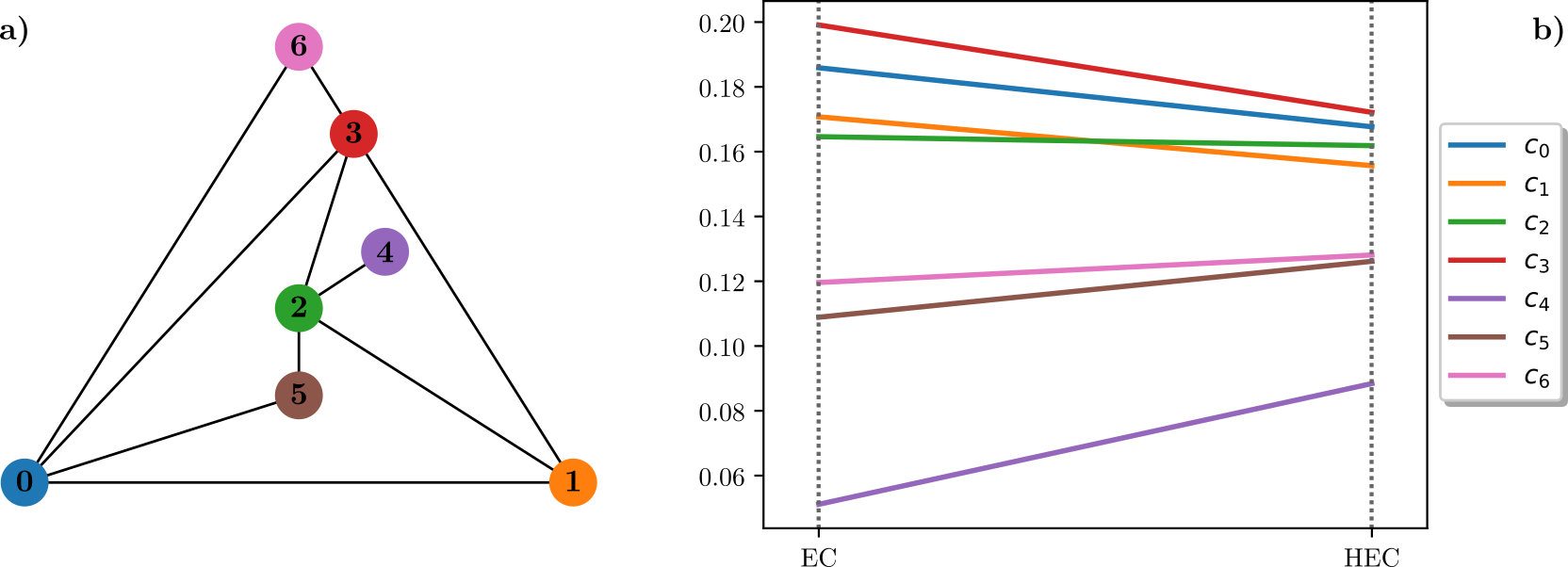}
    \caption{Panel \textbf{a)} shows the pairwise graph $G=(V,E)$ with 7 nodes whose Eigenvector Centrality (EC) and the HEC of its uplift to order 3 (HEC) are calculated. Panel \textbf{b)} shows the centrality values associated to the $\ell_1$-normalized centralities in a parallel coordinate plot. The effect of the homogenization can be clearly seen, as well as the crossing in ranking between nodes $1$ and $2$.}
    \label{fig:pairwise_graphs}
\end{figure}

\subsection{Uplifting and projecting hypergraphs}\label{subsec:both}

So far we have discussed the way to deal with the hyperedges of size lower than the desired one, by means of uplifting those below it. But in order to truly embrace non-uniform hypergraphs we should also consider an operation bringing hyperedges of higher orders down to the desired one. 

The key to this has been hinted at when discussing the Clique motif Eigenvector Centrality, back in Subsection \ref{subsec:art}. There, an order $k$ hyperedge is split into all possible pairwise relations, $C(k,2)= \binom{k}{2} $ of them, between its constituents. In other words, size $k$ edges are projected into sets of size $2$ edges. We can think of an analogous process but turning size $k$ edges into sets of size $p<k$ edges.  

\begin{definition}[Projected hypergraph] \label{def:projection}
    Let $H = (V, E)$ be a hypergraph whose maximum hyperedge size is $M$ and let $2 \leq p < M$. Denote the set of hyperedges of size greater than $p$ as $E'$ and denote the set of all $p$-subsets of every element of $E'$ as $S$. We define the projected hypergraph hypergraph at order $p$ as
    \begin{equation}
        \hH = (V, \hE),\quad  \hE = \left(E \backslash E' \right) \cup S .
    \end{equation}
\end{definition}

In other words, we can break apart each hyperedge of dimension $k$ into $C(k,p) = \binom{k}{p}$ distinct hyperedges of dimension $p$.


Notice that, unlike what we did in the uplift case, we can't as of yet define an associated adjacency tensor, as $\widehat{H}$ will generally still be non-uniform. However, given that this operation entails, essentially, a substitution of each higher size edge by a collection of smaller ones, we need to discuss how to assign weights to the smaller ones generated from the projection. 

If we follow a similar reasoning to the combinatorial one used in the uplift case (see Definition~\ref{def:uplift-tensor}), one ends up with nonsensical weight assignments, particularly it can be calculated to be 
\[
w=\frac {k!}{p!\, C(k,p)}=(k-p)!
\]
For instance, an order 4 hyperedge projected would be projected into order 2 hyperedges with weight $w=2$, hence the resulting hyperedges would have a higher participation than those already at the chosen order. 

Instead, we can go back to Benson's work \cite{benson2019three}, and in particular the CEC calculation, which achieves a sensible projection assigning weights which are the result of counting how many times a pair participates in higher size edges. Our projection aims to generalize this concept, thus the weights come from a similar counting argument  (a $p$-subset's weight will be the number of higher-than-$p$ order edges where the subset participates). 

\paragraph{Joining everything together:}
as we mentioned before, the resulting projected hypergraph $\hH$ might not be uniform, moreover, if we discard orders lower than $p$ we are losing information, as if we uplift hyperedges discarding even higher interactions. For that reason, the key idea in terms of computing centralities is combining both projection of orders higher than a chosen order $p$ and uplifting the lower ones.

\begin{definition}[$p$-UPHEC] \label{def:p-UPHEC}
    Let $H = (V, E)$ be a hypergraph whose maximum hyperedge size is $M$ and let $2 \leq p \leq M$. The $p$-Uplifted-Projected $\cH$-eigenvector centrality ($p$-UPHEC) is the only positive $\cH$-eigenvector of the uniform, weighted hypergraph resulting from 
    \begin{enumerate}
        \item Adding an auxiliary node (or more than on as long as they are indistinguishable) to each hyperedge of size $k < p$ and weighting them with their corresponding combinatorial factor.
        \item Projecting down each hyperedge of size $k > p$ into a set of size $p$ hyperedges, with their corresponding combinatorial factors.
    \end{enumerate}
\end{definition}

As in the UHEC case, for the sake of conciseness we will from now on refer to the $p$-UPHEC as just the UPHEC, where again the order will be clear by the context, or specified otherwise.

It is straightforward to check is the fact that the connectivity of the resulting hypergraph is unchanged.
\begin{lemma}[Strong connectedness of the projected hypergraph] \label{lem:projected-connected}
    Let $H=(V,E)$ be a strongly connected hypergraph and $2 \leq p \leq M$. The hypergraph resulting from uplifting and projecting as in Definition \ref{def:p-UPHEC} is strongly connected.
\end{lemma}

And once again, we can easily show the consistency of this measure, as was the case with the UHEC.
\begin{theorem}[Existence and uniqueness of the UHEC]
    Under the assumptions of Definition \ref{def:m-UHEC}, the UHEC vector exists and it is unique.
\end{theorem}

\begin{proof}
    Lemmas \ref{lem:uplifted-connected} and \ref{lem:projected-connected} guarantee the strong connectedness of the hypergraph resulting from projecting and/or uplifting, and the Perron-Frobenius theorem for strongly connected hypergraphs \cite{chang2008perron} guarantees the existence and uniqueness of its HEC. 
\end{proof}

Note that there may be different UPHEC solutions associated to different values of the parameter $p$. To see this, consider the following example.

\newpage

\begin{figure}[h!]
    \centering
    \includegraphics[width=0.7\textwidth]{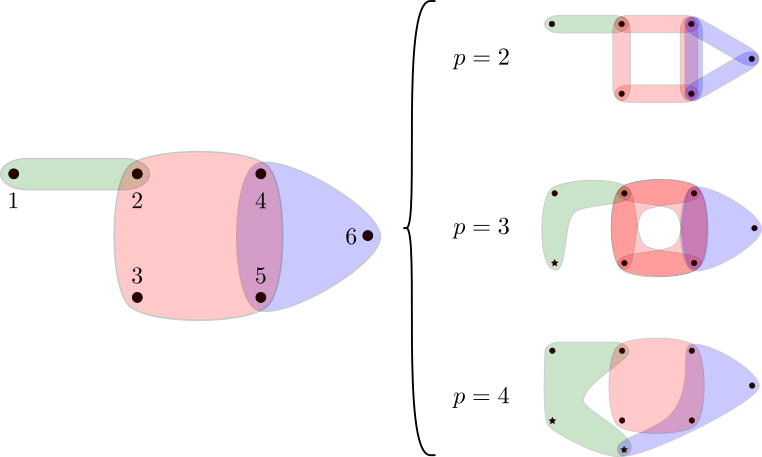}
    \caption{Example hypergraph and its three possible uniformizations at $p=2,3,4$.}
    \label{fig:enter-label}
\end{figure}

\begin{example} \label{ex:UPHEC}
    Let $H=(V, E)$ with $E=\{\{1,2\}, \{2,3,4,5\}, \{4,5,6\}\}$, hence $M=4$. There are three possible UPHEC vectors one can obtain, one for each $p=2,3,4$.
    \begin{itemize}
    
        \item Case $p=2$: This is equivalent to only considering the projection to order 2.
        \begin{equation}
        H' = (V,E'),\enspace \text{with} \enspace E' = \{\{1,2\}, \{2,3\}, \{2,4\}, \{2,5\}, \{3,4\}, \{3,5\}, \{4,5\},\{4,6\}, \{5,6\} \}.
        \end{equation}
        
        \item Case $p=3$: In this case we mix the projection of the second hyperedge and the uplift of the first one, therefore computing the HEC of
        \begin{equation}
        H' = (V',E'),\enspace \text{with} \enspace E' = \{\{1,2,\star\}, \{2,3,4\}, \{2,3,5\}, \{3,4,5\}, \{4,5,6\} \}.
        \end{equation}

        \item Case $p=4$: This is equivalent to only considering the uplift to order 4, i.e. computing the 4-UHEC of
        \begin{equation}
        \wH = (\widetilde{V},\wE),\enspace \text{with} \enspace \wE = \{\{1,2,\star,\star\}, \{2,3,4,5\}, \{4,5,6,\star\} \}.
        \end{equation}
        
    \end{itemize}

    Constructing the respective adjacency tensors and computing their Perron-like $\cH$-eigenvector, we get the normalized centrality scores of Table \ref{tab:example_UHEC}.
    {\renewcommand{\arraystretch}{1.5}%
    \begin{table}[h!]
        \centering
        \begin{tabular}{|C{1.2cm}||C{1.5cm}|C{1.5cm}|C{1.5cm}|C{1.5cm}|C{1.5cm}|C{1.5cm}|}
        \hline
           Case & $c_1$ & $c_2$ & $c_3$ & $c_4$ & $c_5$ & $c_6$ \\
           \hline
           \hline
            $p=2$ & \bf \cellcolor{vermilion} 0.0929 & 0.1802 & 0.1690 & \bf \cellcolor{guppiegreen} 0.2084 & \bf \cellcolor{guppiegreen} 0.2084 & 0.1412 \\ 
            \hline
            $p=3$ & \bf \cellcolor{vermilion} 0.0623 & 0.1949 & 0.1943 & \bf \cellcolor{guppiegreen} 0.2060 & \bf \cellcolor{guppiegreen} 0.2060 & 0.1364 \\
            \hline
            $p=4$ & \bf \cellcolor{vermilion} 0.0853 & 0.1959 & 0.1953 & \bf \cellcolor{guppiegreen} 0.1993 & \bf \cellcolor{guppiegreen} 0.1993 & 0.1250 \\
            \hline
        \end{tabular}
        \caption{Centrality scores for each UPHEC case. Cells highlighted in green show the most central nodes in each case, while cells highlighted in red show the least central nodes. A good consistency check is the fact that the centralities of nodes 4 and 5 are identical in either case, as they are indistinguishable in the original hypergraph $H$.}
        \label{tab:example_UHEC}
    \end{table}
    }

    We can compare this method with the vector centrality one \cite{kovalenko2022vector}, which also distinguishes centralities order by order. The resulting centralities can be found in Table \ref{tab:example_vectorcent}.
    {\renewcommand{\arraystretch}{1.5}%
    \begin{table}[h!]
        \centering
        \begin{tabular}{|C{1.2cm}||C{1.5cm}|C{1.5cm}|C{1.5cm}|C{1.5cm}|C{1.5cm}|C{1.5cm}|}
        \hline
           Order & $c_1$ & $c_2$ & $c_3$ & $c_4$ & $c_5$ & $c_6$ \\
           \hline
           \hline
            $2$ & 0.5 & 0.5 & 0.0 & 0.0 & 0.0 & 0.0 \\ 
            \hline
            $3$ & 0.0 & 0.0 & 0.0 & 0.3333 & 0.3333 & 0.3333 \\
            \hline
            $4$ & 0.0 & 0.25 & 0.25 & 0.25 & 0.25 & 0.0 \\
            \hline
        \end{tabular}
        \caption{Centrality scores yielded by the vector centrality measure at each order.}
        \label{tab:example_vectorcent}
    \end{table}
    }

    Here we see an example where the new measure improves upon existing ones, as it performs a similar task but it is capable of aggregating information of the whole hypergraph structure into each of the evaluated orders, rather than dismissing those which the nodes do not belong to. 
    
    In fact, one can see that one the whole structure is taken into account, in this example there would be no doubt about which is the least important node in the \textit{whole} network and which are two most important ones. If one were to trust the vector centrality at second order, for instance, one could have been deceived into thinking that the first node is of rather remarkable importance. Moreover, the naïve way to combine these orders (summing the scores of each node) would also lead us to think that node 1 is more important than node 3, for example. It should be clear by now that the non-linear treatment is offering us valuable insights.
    
    
\end{example}

\paragraph{Notes on computational complexity:} Before moving to real-world applications, we first want to address the computational cost of the algorithms discussed so far.


Firstly, we need to discuss the creation of the tensor, which will have a different complexity depending on whether we are uplifting or projecting a hyperedge. In the case of the uplift, for every hyperedge that has to be uplifted we add the phantom node the necessary times (linear operation). It gets more complicated in the case of projecting, where we needed to compute all the possible combinations of a hyperedge (factorial operation). Let $H = (V, E)$ by a hypergraph, $m \in \mathbb{N}$ the order we want to transform it to. Let $|E|=|E_u| + |E_m| + |E_p|$, where $|E_u|$ is the number of hyperedges that have to be uplifted, $|E_m|$ is the number of hyperedges already at the desired order and $|E_p|$ is the number of hyperedges that have to be projected. Then, the overall number of operations that have to be done to create this weighted tensor is
\begin{equation}
   \sum_{e_u \in E_u} (m - |e_u|) + \sum_{e_p \in E_p} m \binom{|e_p|}{m} = |E_u|(m - \bar{e}_u) + \sum_{e_p \in E_p} m \binom{|e_p|}{m},
\end{equation}
with $\bar{e}_u$ being the average size of hyperedges that have to be uplifted. 
To compute the Big-O notation we have to choose the worst case scenario, the highest order term. In this case, it will be that associated to the projected edges
\begin{equation}
    \mathcal{O}\left( |E| \cdot m \cdot \binom{ |e| }{m} \right).
\end{equation}

Once we have created the tensor, we now need to compute the eigenvector corresponding to the largest $\cH$-eigenvalue. In order to compute UHEC and UPHEC centralities, instead of creating a new algorithm, we have used a variant of the power method with a weighted tensor (see \cite{ng2009finding}). 

%% file: HEC-numerical.tex

A first attempt to generalize adjacency tensors in a non-uniform context was provided by \cite{banerjee2014spectra} (and later glossed over by Benson in \cite{benson2019three}), which goes by the name hyperedge ``blowups'' \cite{aksoy2024scalable}. This method relies on suitably duplicating indices in the adjacency tensor to accommodate to higher order hyperedges, and it has recently been computationally improved so as to avoid its high computational cost when it comes to the tensor apply operation \cite{aksoy2024scalable}. However, and as \cite{qi2017tensor} already points out, there is some indeterminacy in this approach.

We will nevertheless consider the original (and only) proposal \cite{banerjee2014spectra} which, given a hyperedge $e=\{v_1\dots v_s\}$ with $2 \leq s \leq m$ nodes (where $m$ is the maximum cardinality of the hyperedges), assigns it the $m$-uniform adjacency tensor components
\begin{equation}\label{eq:alternative_uniformization}
    a_{i_1 \dots i_m}= \frac{s}{\alpha} \quad \text{where} \quad \alpha = \sum_{p_1,\dots p_r=1}^n \frac{m!}{p_1! p_2! \dots p_s!}.
\end{equation}
and $i_1,\dots, i_m$ are chosen in all possible ways from $\{v_1,\dots, v_r\}$. The construction of this tensor is already disadvantageous. Time complexity of this uplifting method can be directly found by the intuitive idea behind it. Let's say we want to uplift the hyperedge $e$ to order $m$. To do so, we will need all the possible combinations of adding each node to it, until we reach the desired order. Increasing the order by 1 would take $|e|$ operations (add each node to the hyperedge once). Increasing the order by 2, we would need to do $|e^2|$ operations (the mentioned before, and for each new hyperedge constructed, add each of the original nodes). It's straightforward that the time complexity we are talking about is $\mathcal{O}\left(|e|^{m-|e|}\right)$ for each hyperedge to be uplifted. Nevertheless, this time complexity can be reduced through dynamical optimization to $\mathcal{O}\left(|e|(m-|e|)\log m\right)$. The method proposed in this paper to uplift a hyperedge involves far fewer operations, having $\mathcal{O}\left({m-|e|}\right)$ for each hyperedge, as the only thing it is being done is adding a new node the necessary times. Moreover, this alternative uniformization does not include a notion of projection, which is why we have to supplement it with ours if one is interested in checking intermediate orders.


We now want to give a flavour of the difference between the different tensorial methods discussed throughout this manuscript, namely: the standard HEC (equation \ref{eq:3uniform-hec}), the UPHEC (Definition~\ref{def:p-UPHEC}) and the alternative uniformization method (equation~(\ref{eq:alternative_uniformization})), at each of the different orders present in a hypergraph\footnote{We will not be discussing the ZEC here, as it was already done in \cite{benson2019three} and it does not have an UPHEC analogue, as discussed throughout the text.}, first in the case of real data and later for synthetically-generated hypergraphs.

\subsubsection{Real-world hypergraph datasets:}

As further evidence of the interest and usefulness of the proposed method, some real world hypergraph datasets are  taken now  into consideration, by analyzing three different points: firstly, how do the two hypergraph uniformizations discussed (our UPHEC and blowups \cite{banerjee2014spectra,aksoy2024scalable}) compare against the original, order-by-order analysis of the hypergraph put forward by Benson \cite{benson2019three}. Secondly, what is the difference between both uniformizations in these real cases. And thirdly, even within either of these uniformizations, one needs to choose at which order to perform the analysis (uplifting lower orders and projecting higher ones in our method, ``blowing up'' lower orders and also projecting higher ones in the blowup one.

It is important to note that the besides the figures shown in the present manuscript, which have been picked due to their clarity and aid in the exposition, we have perform a wider analysis of more hypergraphs (all of them freely available within the XGI library \cite{landry2023xgi}), which can be found in the open repository at \url{https://github.com/LaComarca-Lab/non-uniform-hypergraphs}.

To start things off, let us consider two hypergraphs: a quintessential one, the \texttt{tags\_ask\_ubuntu} dataset \cite{benson2018simplicial}, also used in \cite{benson2019three} to showcase the CEC, ZEC, and HEC proposals, and the \texttt{hypertext-conference} one \cite{isella2011what}. The former contains information about interactions within the \textit{Ask Ubuntu StackOverflow} forum. Specifically, it can be seen as a hypergraph where nodes represent tags and hyperedges between tags represent questions asked marked with those tags. The latter contains data gathered during the \textit{ACM Hypertext 2009} conference, pertaining the interactions between its participants.

Some basic statistics of these hypergraphs (after pre-processing them with the XGI library \cite{landry2023xgi} in order to remove isolated nodes, singleton edges, etc) can be observed in Table~\ref{tab:askubuntu}. Note that when studying each uniform order as isolated some nodes will become disconnected if they have no such interactions.

{\renewcommand{\arraystretch}{1.5}%
\begin{table}[h!]
    \centering
    \begin{tabular}{|c|ccc|ccc|}
        \hline
         & \multicolumn{3}{c|}{\texttt{tags\_ask\_ubuntu} \cite{benson2018simplicial}} & \multicolumn{3}{c|}{  \texttt{hypertext-conference} \cite{isella2011what}} \\
        \hline
        Order & Nodes & Hyperedges & Size of LCC & Nodes & Hyperedges & Size of LCC\\
        \hline
        \hline
        2 & 2714 & 28134 & 89.84\%  & 113 & 2103 & 100\%\\
        \hline
        3 & 2821 & 52282 & 93.38\% & 105 & 302 & 92.92\%\\
        \hline
        4 & 2722 & 39158 & 90.10\% & 11 & 12 & 9.73\%\\
        \hline
        5 & 2564 & 25475 & 84.87\% & 8 & 7 & 7.08\%\\
        \hline
        6 & - & - & - & 8 & 4 & 7.08\%\\
        \hline
        \hline
 Complete & 3021 & 145053 & 100\% & 113 & 2434 & 100\%\\
        \hline
    \end{tabular}
    \caption{Number of nodes and hyperedges at each order of the hypergraphs, as well as the size of the Largest Connected Component (LCC) containing them compared to the total. There is a small discrepancy between the sum of hyperedges at each order and the total number, due to the fact that some hyperedges become disconnected if we remove other orders and therefore they won't appear in the LCC}
    \label{tab:askubuntu}
\end{table}
}

The natural way to compare rankings is by means of some correlation measure which only takes into account the ordinal correlation between the entries (i.e. their position within the ranking) rather than their actual magnitudes. One of the best known examples of this measure is Kendall's tau correlation coefficient ($\tau\in[0,1]$, where the closer to 1 the more correlated), which we will compute between every pair of rankings.

Before showing the actual results, we should mention that in order to compare two rankings, they must contain the same number of elements. However in the uniformized vs non-uniformized cases this is not the case (the non-uniformized, i.e. standard HEC versions only keep the LCC with those interactions). For that reason we have chosen to fill the empty entries with a zero value, as they do not participate in such order. It is here that we can already glimpse at the issue with the standard, non-uniformized HEC: if we look at Table~\ref{tab:askubuntu} we can see that, while this may be sensible in the \texttt{tags\_ask\_ubuntu} dataset, in the \texttt{hypertext-conference} the LCC of order beyond 3 is a minuscule part of the entire hypergraph. For this reason, the ranking will be localized around those nodes, yielding $\tau\approx 0$.

The results of the comparison between each of the rankings are shown in Figure~\ref{fig:KT-double_heatmap}. At this stage, we will focus on the first of the questions described before: the comparison of either uniformization with the non-uniform approach per order.

\begin{figure}[h!]
    \hspace{-0.7in}
    \includegraphics[width=1.25\textwidth]{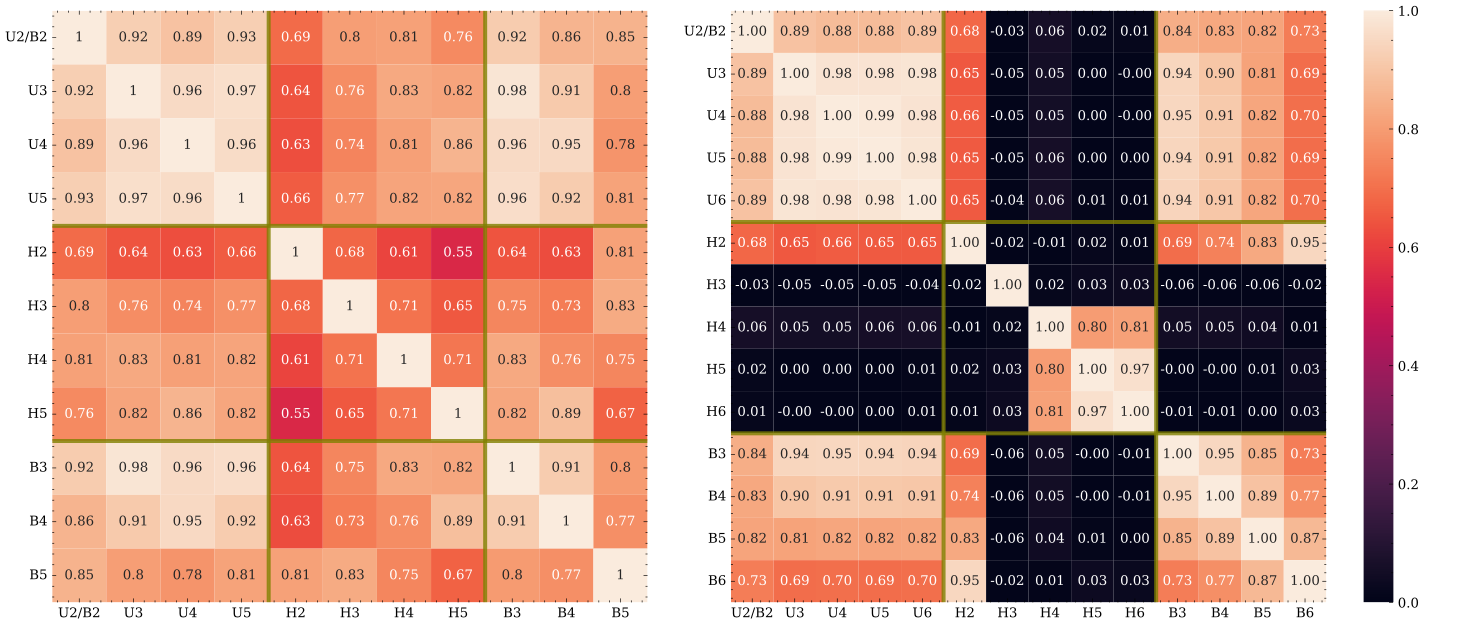}
    \caption{Kendall's tau correlation coefficient between the whole rankings obtained in each of the method, for the \texttt{tags\_ask\_ubuntu} (left) and \texttt{hypertext-conference} (right) datasets. Methods are labelled as U2, U3, U4, U5 for the UPHEC case; H2, H3, H4, H5 for the HEC at each order, and B3, B4, B5 for the blowup uniformization discussed above. U2 and B2 are equal as they only include the projection, and are therefore shown together.}
    \label{fig:KT-double_heatmap}
\end{figure}

In the \texttt{tags\_ask\_ubuntu} dataset, the four standard HEC measures among themselves have the lowest correlations. The lowest correlation of the whole Figure is actually that between these 2nd and 5th orders. This is product of the fact that the uniform hypergraphs at each order have little to do with each other, they each describe a portion of the whole.

As we advanced before, this is much more drastic in the \texttt{hypertext-conference} dataset: there almost every correlation yields a number close to zero, except for the one between orders 5 and 6, as the LCC's of these orders share the same 8 nodes (see Table~\ref{tab:askubuntu}).

This analysis makes it clear that there is a need for uniformized versions of the HEC centrality, as the order-by-order study of hypergraphs clearly lacks a cohesive description of the whole\footnote{ For a visual analogy of what is going on, check the cover of the first edition of ``Gödel, Escher, Bach: an Eternal Golden Braid'', by Douglas R. Hofstadter \cite{hofstadter1979godel}.}

Having dealt with the question of uniform measures versus non-uniform ones, we now shift our focus to the other two problems: the comparison between uplift and blowups, and the order to inspect. In order to get a better understanding, we supplement the previous examples with other four real hypergraph datasets, also available in XGI, whose most basic statistics (this time without an order-by-order overview) after preprocessing are summarized in Table \ref{tab:other-datasets}.
{\renewcommand{\arraystretch}{1.5}%
\begin{table}[h!]
    \centering
    \begin{tabular}{|c|ccc|cc|}
        \hline
        Dataset &  Nodes & Hyperedges & Maximum order & $\langle\tau_{UU}\rangle$ & $\langle\tau_{BB}\rangle$\\
        \hline
        \hline
        \texttt{tags\_ask\_ubuntu} \cite{benson2018simplicial} & 3021 & 145053 & 5 & 0.960 & 0.825 \\
        \hline
        \texttt{hypertext-conference} \cite{isella2011what} & 113 & 2434 & 6 & 0.982 & 0.844 \\
        \hline
        \texttt{contact-primary-school} \cite{stehle2011high} & 242 & 12704 & 5 & 0.962 & 0.905 \\
        \hline
        \texttt{contact-high-school} \cite{mastrandrea2015contact} & 327 & 7818 & 5 & 0.946 & 0.863 \\
        \hline
        \texttt{sfhh-conference} \cite{cattuto2010dynamics} & 403 & 10541 & 9 & 0.918 & 0.748 \\
        \hline
        \texttt{diseasome} \cite{goh2007human} & 516 & 314 & 11 & 0.724 & 0.590 \\
        \hline
    \end{tabular}
    \caption{Number of nodes, hyperedges and maximum order of each hypergraph (after removing isolated nodes and duplicated edges and keeping the LCC), as well as average Kendall-tau coefficient between the UPHECs $\langle\tau_{UU}\rangle$ and the blowups $\langle\tau_{BB}\rangle$ at each order between 3 and the respective maximum order.}
    \label{tab:other-datasets}
\end{table}
}

\newpage

For each of these hypergraphs, the correlation between the UPHEC and blowup+projection measures have been computed  in Figure~\ref{fig:KT-multi_heatmap}, now ignoring the non-uniform measures for ease of visualization\footnote{In the \texttt{diseasome} case we have also restricted the computation to orders up to 9 for better visualization.}, as well as the the top-most column and left-most row (the ``U2/B2'' ones in Figure \ref{fig:KT-double_heatmap}, for they correspond to only projecting any higher-order interaction to pairwise ones, and computing the HEC of the corresponding, uniform hypergraph. As there are no uplifts nor blowups, there is no distinction on the uniformization used, which is why we choose to ignore them at this point.

\begin{figure}[h!]
    \hspace{-0.5in}
    \includegraphics[width=1.15\textwidth]{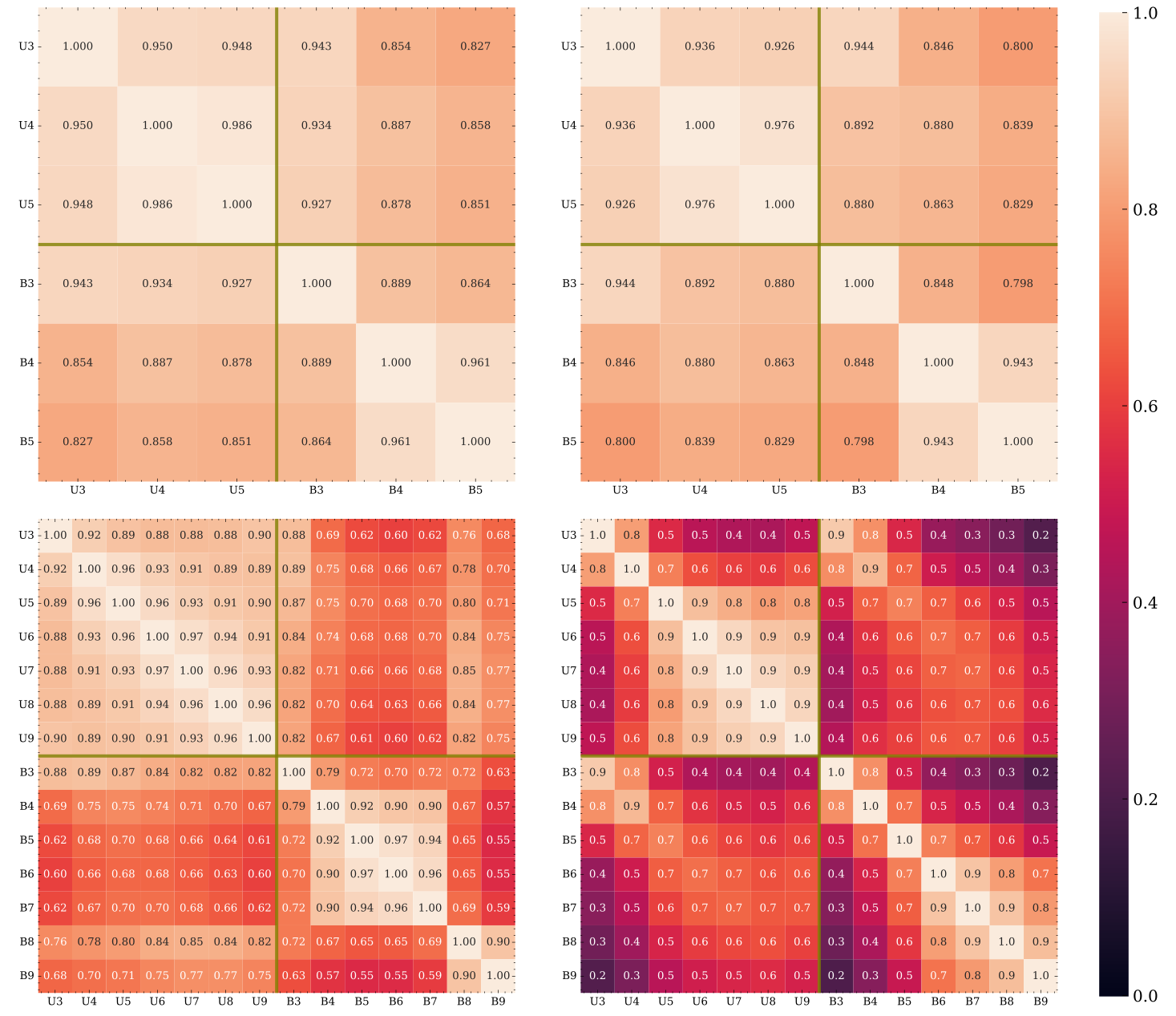}
    \caption{Kendall's tau correlation coefficient between the whole rankings obtained by the UPHEC and blowup uniformizations (with the same labelling conventions as in Figure~\ref{fig:KT-double_heatmap}) for the \texttt{contact-primary-school} (top-left), \texttt{contact-high-school} (top-right), \texttt{sfhh-conference} (bottom-left) and \texttt{diseasome} (bottom-right) datasets.}
    \label{fig:KT-multi_heatmap}
\end{figure}

We can draw the following conclusions from Figures~\ref{fig:KT-double_heatmap} and \ref{fig:KT-multi_heatmap}, with respect to both uniformization procedures.

Firstly, the average correlation between the same type of uniformizations at different orders (i.e. the U-U, B-B quadrants) is always higher in the uplift than in the blowup cases. For reference, the average value of off-diagonal correlations in those two quadrants, for each hypergraph, are shown in Table~\ref{tab:other-datasets}. In that regard, it is interesting to look at in the \texttt{sfhh-conference} example: the lowest correlation between any two UPHEC measures is found to be around 0.88, while the lowest in the blowup uniformization is around 0.55.

Moreover, one can clearly see that the higher the order inspected, the more disparity. This is pointing out to the fact that, while in lower orders the projection part of the measure (which is the same in both the UPHEC and in the blowup uniformization) is evening the rankings, when we focus on the highest order (thus only having vanilla uplift and blowup, no projection) the blowup is computing something slightly different. In that sense, this seems to confirm the claim in \cite{qi2017tensor} about the blowup uniformization and the fact that it contains a degree of arbitrariness in the augmentation, something which indeed observe.

Apart from the choice of uniformization procedure, we wanted to understand the implications of the choice of order at which to inspect the hypergraph. Focusing on the UPHEC method, what we can see is that in most examples the choice is basically irrelevant: once we take into account every level of interaction (either through projection or uplift), a centrality unison emerges, something we can clearly see from $\langle\tau_{UU}\rangle \approx 1$. Nevertheless, the correlation is better at higher orders, meaning that the more uplift and less projection, the more agreement in the description of the overall hypergraph.

At this point it is important to also consider the computational cost of each of these methods. As we have discussed, ideally one would want to compute the centrality with UPHEC at the highest order available. However, it might be preferable to have a balance between uplift and projection, staying therefore at intermediate orders. Alternatively, and if computational efficiency is a necessity, one could use the method proposed in \cite{aksoy2024scalable}, which achieves a remarkable speed increase in the computation of the blowup, turning an $\mathcal{O}(n^M)$ problem into one being polynomial in $M$, the maximum order.






Apart from the full ranking comparison, it is often interesting to understand how does the correlation change when we contrast the top $K$ nodes obtained with a method with their corresponding ranking according to another method, as we increase the amount $K$ of nodes sampled. For the sake of simplicity we will show this with the \texttt{tags\_ask\_ubuntu} case, as the results are similar in the rest of the datasets.

Given the amount of possible comparisons (12 in the case of UPHEC-UPHEC, 16 in the cases of UPHEC-HEC, etc), we have decided to filter out most of them in order to present a meaningful figure. In particular, for each measure comparison we have chosen to keep at most correlations: the correlation reaching the highest maximum, the correlation reaching the lowest minimum, and the two correlations whose average is minimum and maximum. We feel that these conditions will provide us with a set of correlations which can convey more information (in the sense of most similar and dissimilar rankings). The resulting plot is displayed in Figure~\ref{fig:KT-topK}, and the unfiltered one can be found in the open repository available at \url{https://github.com/LaComarca-Lab/non-uniform-hypergraphs}.

\begin{figure}[h!]
    \hspace{-0.3in}
    \includegraphics[scale=0.63]{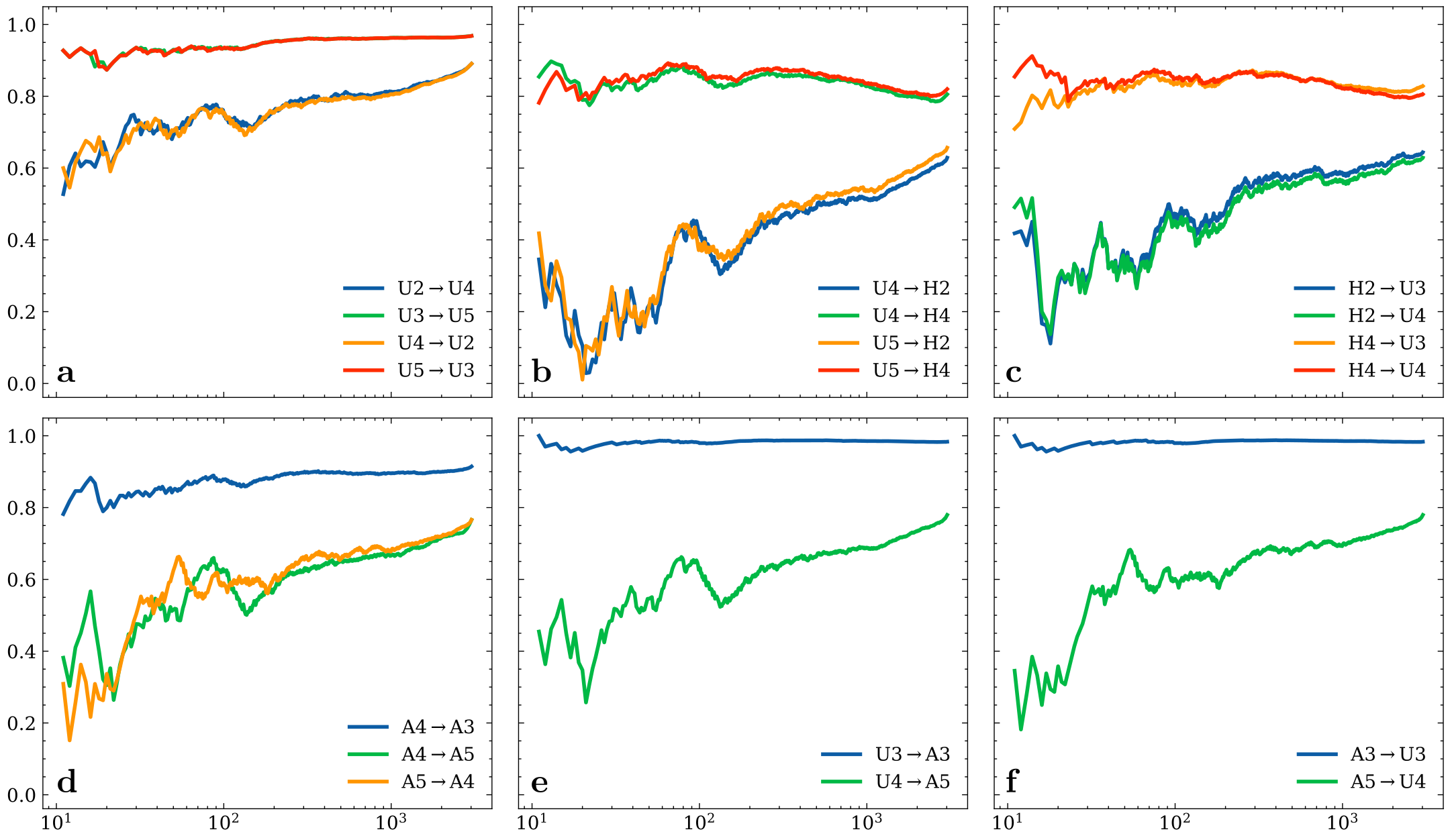}
    \caption{Kendall's tau correlation coefficient (log scale) between the top $K$ ranked nodes from one method with the same nodes from another method on the Ask Ubuntu co-tagging dataset. Panels \textbf{a}, \textbf{b}, \textbf{c}, \textbf{d}, \textbf{e} and \textbf{f} represent the UPHEC-UPHEC, UPHEC-HEC, HEC-UPHEC, Alternative-Alternative, UPHEC-Alternative and Alternative-UPHEC, respectively.}
    \label{fig:KT-topK}
\end{figure}

We see that despite some initial fluctuations around $K=100$, most correlations tend to increase or stabilize, converging to their respective values shown in Figure~\ref{fig:KT-double_heatmap}. We also notice that in most cases the minimums are reached in pairs, e.g. U2 and U4 are not very correlated with each other in Subfigure 1\textbf{a} in either direction, which is rather sensible.

\subsubsection{Synthetic hypergraphs}
In this subsection we introduce some experiments that will help confirming the robustness of the proposed method. To do so, we use several synthetic hypergraphs, not only as a source of information, but also to prove that the new proposed measure works as expected in any kind of hypergraph, regardless of the domain to which the hypergraph belongs.

We use the natural extension of the Erd\H{o}s-R\'enyi graph generation model to nonuniform hypergraphs, as provided by \cite{dewar2018}. In the pairwise case, this method works by choosing $p_2$, which expresses the probability of any two nodes $v_1$ and $v_2$ forming the edge $(v_1, v_2)$. In the extension, to generate a \textit{random non-uniform hypergraph} with $n$ nodes and hyperedge sizes in $\{2, ..., m\}$, we provide probabilities $(p_1, ..., p_m)$, where each $p_i$ expresses the probability of forming $i$-hyperedges between any $i$ nodes in the hypergraph, and generate them accordingly. 

The generated hypergraphs used here have $n = 100$ nodes and the hyperedge sizes range from 2 to 5. We selected $p_2 = 0.1 > \log(n)/n$ such that we can ensure that the generated hypergraph is strongly connected and $p_5 = 10^{-6}$ such that each generated hypergraph has approximately $100$ $5$-hyperedges.

By using fixing these parameters, we iterate through $p_3$ and $p_4$, assuring that we would have around $150 \text{ and }200$ $3$ and $4$-hyperedges as the minimum, and around $1500 \text{ and } 2000$ as the maximum, respectively. For this purpose, both $p_3$ and $p_4$ were taken from two equally spaced ranges, each split into eight equally spaced parts with $p_3 \in (10^{-3}, 10^{-2})$ and $p_4 \in (5\cdot 10^{-5}, 5\cdot 10^{-4})$.

In this way, we have 64 possible combinations of $(p_3, p_4)$ to generate hypergraphs. We generated $20$ hypergraphs for each $4$-tuple $(p_2, p_3, p_4, p_5)$. For each of these 1280 hypergraphs, we computed the $k$-UPHEC with $k \in \{2, 3, 4, 5\}$ and computed Kendall's tau for all combinations of measures, as shown in Figure \ref{fig:synhtetic-correlation}.

 \begin{figure}[h]
     \centering
     \includegraphics[width=1\textwidth]{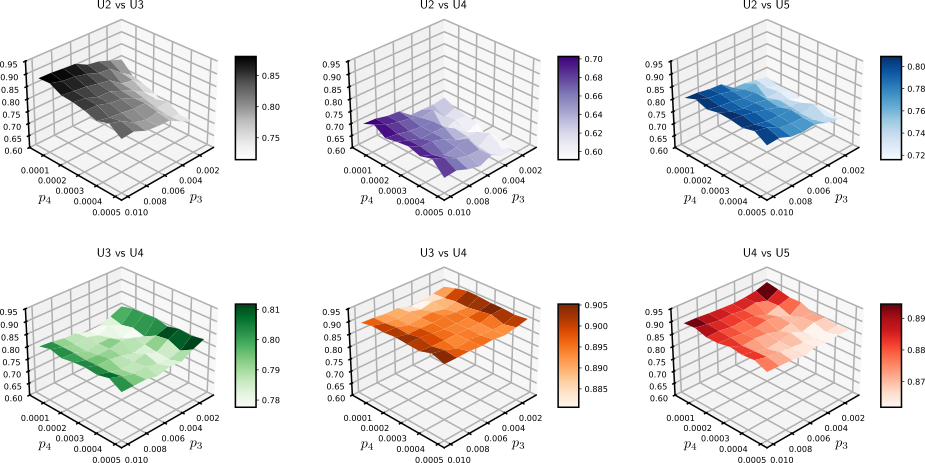}
     \caption{Average Kendall's tau correlation coefficient between the rankings obtained for each of the generated hypergraphs. ``U$x$ vs U$y$'' represents the correlation between the UPHECs at orders $x$ and $y$.}
     \label{fig:synhtetic-correlation}
 \end{figure}

As we can see in Figure \ref{fig:synhtetic-correlation}, the correlations can be split into two disjoint behaviors, the upper and lower rows. The upper row shows the proportional growth in the correlation with $p_3$. We discuss the effects of $p_4$ on the different images in this row. Having noted that both $p_2$ and $p_5$ are fixed, if we compare ``U2 vs U3'', the growth of $p_3$ is proportional to the growth of the correlation. One can easily identify that the reason behind this is that having more $3$-hyperedges will compensate for the difference in size between $4$ and $5$-hyperedges and $2$-hyperedges. In addition, having more $3$-hyperedges makes the $3$-UPHEC closer to the original hypergraph. The opposite occurs in the case where we introduce more $4$-hyperedges; we will not only have the difference between the $2$-hyperedges, but also with $3$-hyperedges.
    
Note that in the case ``U2 vs U4'' we can only see a proportional growth with $p_3$, but not the inverse proportional growth with $p_4$ as much as before. It is trivial to compute the $4$-UPHEC; introducing $4$-hyperedges will have a positive effect, as it introduced $3$-hyperedges in the previous case. Here, $p_4$ also compensates for the effect of $2$ and $4$-UPHEC on the $5$-hyperedges, as does (and did before) $p_3$. Finally, if in the ``U2 vs U5'' case,  the effect of $p_4$ appears to be irrelevant, but we detect a strong effect of $p_3$, as $3$-hyperedges are again acting as an intermediate between the $2$ and $5$-UPHEC.


Now that we have discussed the obtained computations on synthetic hypergraphs, we can conclude that regardless of the differences for different pairs $(p_3, p_4)$, the proposed uniformization keeps track of the centrality measures, where each node has a similar ranking position in the different uniformizations; that is, the importance of each node is relatively preserved.

%% file: ZEC-uplift.tex
As we discussed previously, the ZEC centrality can't be computed when uplifting a non-uniform hypergraph, as the different sums can only be grouped together if we are able to rescale the centrality such that $c_\star=1$, which we can't if we are considering $\cZ$-eigenvectors. However, this is not an issue if we start from a $2$-uniform hypergraph (in other words, a pairwise graph).

In fact, $\cZ$-eigenvectors allow us to generalize the uplift operation to having more than one different auxiliary node\footnote{This generalization was already possible in the HEC case, however in that case it only cluttered the notation and hampered the calculation, as the computational complexity scales with the number of \textit{distinct} nodes involved. Note also that in that case further conditions would be required for a well-defined uplift, as in order to be able to scale the centrality such that $c_{\star_k}=1\,\forall k$, we need all of them to be indistiguishable from each other, i.e. they must be related by permutation.}, e.g. $\star_1,...,\star_k$.

\begin{definition}[Multi-Uplifted hypergraph]
Let $H = (V, E)$ be an $M$-uniform hypergraph and let $m \geq M$. We define the multi-uplifted hypergraph at order $m$ with $s$ auxiliary nodes, each contained $p_k$ times within each hyperedge as
\begin{equation}
    \widetilde{H}=(\widetilde{V}, \wE), \quad \text{where} \quad \widetilde{V}=V\cup \{\star_1, ..., \star_s\} \quad \text{and} \quad \widetilde{E}= \left\{ e \cup \left( \bigcup_{k=1}^{s}  \bigcup_{l=1}^{p_k}  \{\star_k\} \right), \, e\in E \right\},
\end{equation}
with $\displaystyle\sum_{k=1}^s p_k = m - M$.
\end{definition}


As we previously declared, this operation on 2-uniform (standard) graphs allow us to relate the $\cZ$-eigenvectors of the adjacency tensor of hypergraphs to those of their original, standard graph. To see this, consider adding two different auxiliary nodes $\star, \times$ to a graph $G$ with adjacency matrix $A=(A_{ij})$. This operation translates into the following rewriting:
\begin{equation}
    \sum_{j=1}^n A_{ij} c_{j} \rightarrow \sum_{j,k,l=1}^{n,\,\star,\times} \wT_{ijkl} c_{j} \, c_{k} \, c_{l} = {\binom{3}{2}} \sum_{k,l=1}^{n} \wT_{ij\star\times} c_{j} c_{\star} \, c_{\times},
\end{equation}
where the notation $\sum_{i=1}^{n,\star,\times}$ indicates summing over the index $i\in [ n ] \cup \{\star,\times\}$.

Given that, by definition, $\wT_{ij\star\times}= \frac{1}{12} A_{ij}$, the $\cZ$-eigenvector equation of the uplifted 4-uniform hypergraph is equivalent to the $\cZ$-eigenvector equation of the original 2-uniform hypergraph, which reduces to the standard eigenvector centrality of the graph:
\begin{equation}
    \lambda \bc = \wT \bc^3, \quad \bc = (c_{1},...,c_n,c_{\star},c_{\times})^T  \Longleftrightarrow \lambda' \bc' = A (\bc')^2, \quad \lambda'=\frac{4\lambda}{c_{\star} c_{\times}}, \quad \bc' = (c_{1},...,c_{n})^T.
\end{equation}

We can extrapolate this example to the uplift from a $2$-uniform hypergraph to an $(2+l)$-uniform hypergraph, as stated in the following theorem.

\begin{theorem}[Correspondence between $\cZ$-eigenvectors] \label{thm:Zcorrespondence}
    Let $\wH=(\widetilde{V},\widetilde{E})$ be a strongly connected, $(2+l)$-uniform hypergraph with $l \geq 1$. If there is a non-empty subset of nodes $V^\star=\{\star_{1}, ... \star_{s}\} \subset \widetilde{V}$, each contained $\{p_1,...,p_s\}$ times, respectively, in every hyperedge, such that $\sum_{i=1}^s p_{i}=l$, then, 
    \begin{itemize}
        
        \item The components of the $\cZ$-eigenvectors of $\widetilde{H}$ associated to the nodes $n\in V= \widetilde{V} \backslash V^\star $  correspond to those of the $2$-uniform hypergraph $H=(V,E)$ having those $s$ nodes removed.
                
        \item The components of the positive $\cZ$-eigenvectors of $\widetilde{H}$ associated to the auxiliary nodes $n\in V^\star$  are uniquely determined by the other components.
        
        \item The $\cZ$-eigenvalues $\widetilde{\lambda}$ of $\wH$ correspond to those of the $2$-uniform hypergraph $H$, $\lambda$, rescaled as
        \begin{equation}
            \wl = \lambda \Omega (\widetilde{c}_{\star_1})^{p_1} ... (\widetilde{c}_{\star_s})^{p_s}, \quad \Omega =   \frac{(l+1)!}{\displaystyle\prod_{i=1}^s (p_i!)}.
        \end{equation}

    \end{itemize}
\end{theorem}

\begin{proof}
    Under the conditions stated, $\widetilde{H}$ can be viewed as an uplift of the hypergraph $H$ with $s$ auxiliary nodes, each one contained equally in each and every hyperedge. The $\cZ$-eigenvector equation for the uplifted hypergraph can be written as
    \begin{align}
        \wl \widetilde{c}_{i_1} = \sum_{i_2,...,i_{2+l}=1}^{n,\,\star_1,...,\star_s} \wT_{i_1,...,i_{2+l}}\, \widetilde{c}_{i_2} \, ...  \, \widetilde{c}_{i_{2+l}} = \Omega \sum_{i_2=1}^{n} T_{i_1,i_2}\, \widetilde{c}_{i_2} \,(\widetilde{c}_{\star_1})^{p_1} ... \, (\widetilde{c}_{\star_s})^{p_s},
    \end{align}
    where we have summed over the auxiliary nodes, recovering the pre-uplifted tensor $T$ components times a combinatorial factor $\Omega$, product of the symmetry of the adjacency tensor. We now carefully calculate this factor.
    \begin{enumerate}
        \item In the equation for node $i$, there will be a sum over $2+l-1=l+1$ indices (1 corresponds to its real $j$-th neighbor, $l$ to the auxiliary nodes added). We will have all possible $(l+1)!$ permutations.

        \item We need to subtract the repetitions of auxiliary nodes, given by their multiplicities $p_i$.
    \end{enumerate}
    Having both of these facts considered we can easily calculate it to be
    \begin{equation}
        \Omega =   \frac{(l+1)!}{\displaystyle\prod_{i=1}^s (p_i!)}.
    \end{equation}
    
    The centralities of these auxiliary nodes can now be pulled out of the sum, obtaining 
    \begin{equation}
        \lambda \widetilde{c}_{i_1} = \sum_{i_2=1}^{n} T_{i_1,i_{2}} \widetilde{c}_{i_2}, \quad \wl = \lambda \Omega (\widetilde{c}_{\star_1})^{p_1} ... (\widetilde{c}_{\star_s})^{p_s},
    \end{equation}
    where we have already re-scaled the $\cZ$-eigenvalue accordingly. Noticing that $T=A$ with $A$ being the adjacency matrix of $G$, we arrive at the equation $\lambda \bc = A\bc$, which is precisely the eigenvector equation of the 2-uniform hypergraph (pairwise graph) $G$. 

    Therefore, the first $n$ components of the $\cZ$-eigenvector of the uplifted hypergraph $\wH$ correspond to the eigenvector $\bc$ associated to $G$.

    It is left for us to discuss the behavior of the remaining equations, one per auxiliary node. Without loss of generality, we consider the equation of node $\star_1$:
    \begin{equation}
        \wl \widetilde{c}_{\star_1} = \sum_{i_2,...,i_{2+l}=1}^{n,\,\star_1,...,\star_s} \wT_{\star_1,...,i_{2+l}}\, \widetilde{c}_{i_2} \, ...  \, \widetilde{c}_{i_{2+l}} = \Omega \sum_{i_1,i_2=1}^{n} T_{i_1,i_2}\, \widetilde{c}_{i_1} \, \widetilde{c}_{i_{2}} \,(\widetilde{c}_{\star_1})^{p_1 - 1} ... \, (\widetilde{c}_{\star_s})^{p_s}.
    \end{equation}

    Multiplying both sides by $c_{\star_1}$ and then replacing $\wl$ in term of $\lambda$, we obtain the following expression
    \begin{equation}
        \lambda (\widetilde{c}_{\star_1})^2 = \sum_{i_1,i_2=1}^{n} T_{i_1,i_{2}} \widetilde{c}_{i_1} \,\widetilde{c}_{i_2}.
    \end{equation}

    Therefore, each component associated to an auxiliary node is uniquely determined (up to a sign, although we can always choose the positive solution) by the components of the non-auxiliary nodes. 
\end{proof}

\begin{remark}
    We have omitted the norm constraint required for $\cZ_1$ or $\cZ_2$-eigenvectors. We are allowed to do so because we uplift a pairwise graph: eigenvectors of the adjacency matrix can be re-scaled as will, therefore the first $n$ components of the $\cZ$-eigenvector of the uplifted hypergraph $\wH$ can be matched to a specific scaling of the eigenvector of the adjacency matrix of $G$.

    Note that this is the reason why this theorem can't be generalized to an uplift from an $m$-uniform hypergraph to an $(m+l)$-uniform hypergraph: even though the $\cZ$-eigenvector equations can be related to each other, in general their norm constraints will be incompatible.
\end{remark}


For illustrative purposes we provide an example which can be analytically solved, following the uplift on Figure \ref{fig:Z-uplift}.

\begin{figure}[h!]
    \centering
    \includegraphics[width=0.6\textwidth]{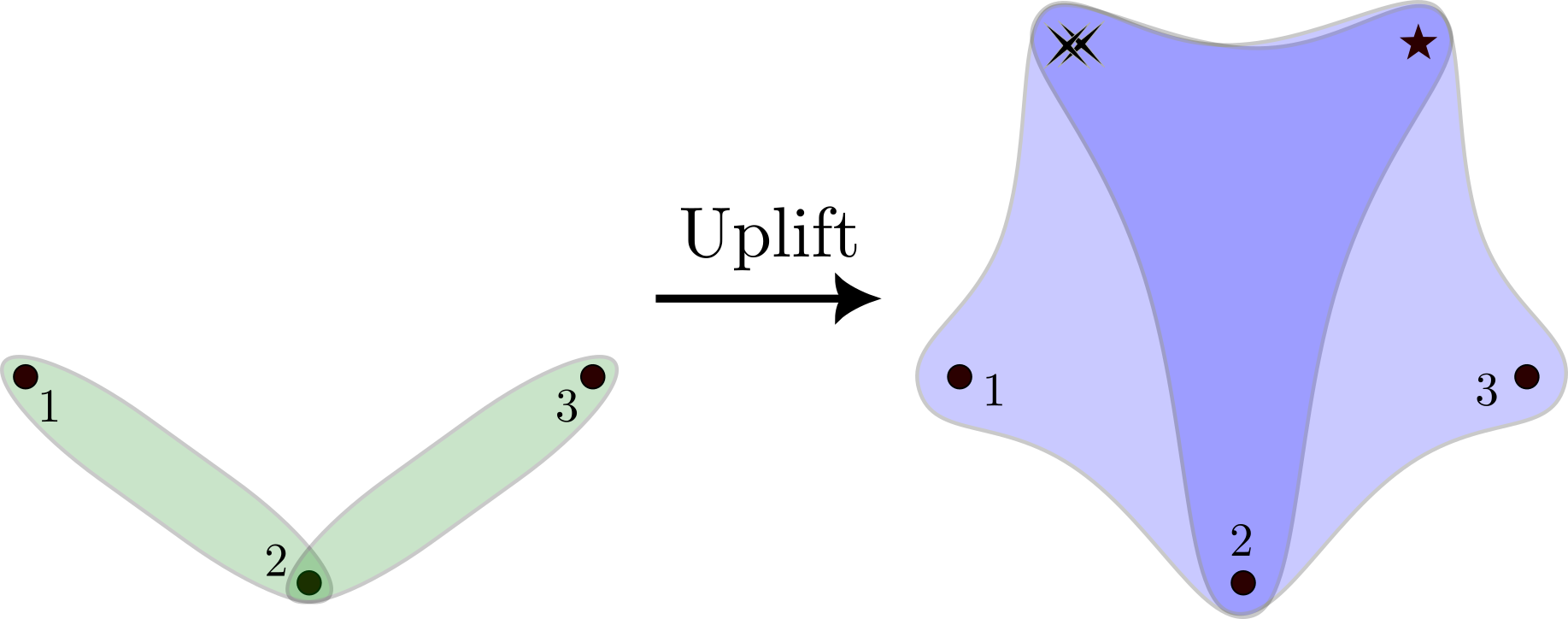}
    \caption{Uplift of a 2-uniform hypergraph (a pairwise graph) to 5-uniform by adding two nodes $\star,\times$ to each hyperedge, the latter being added twice (depicted as there being two of them for illustrative purposes).}
    \label{fig:Z-uplift}
\end{figure}

\begin{example}
Consider the graph $G=(V,E)$ with nodeset $V=\{1,2,3\}$ and edgeset $E = \{\{1,2\},\{2,3\}\})$. It can be seen as a 2-uniform hypergraph $H \simeq G$. Suppose we uplift it to a 5-uniform hypergraph $\wH$, adding auxiliary nodes $\star$ and $\times$ in each hyperedge, the former once and the latter two times, i.e.
\begin{equation}
    \wH = (\widetilde{V}, \widetilde{E}), \quad \widetilde{V} = \{1,2,3,\star,\times\}, \quad \widetilde{E} = \{\{1, 2, \star, \times, \times\}, \{2, 3, \star, \times, \times\}\}, 
\end{equation}
as shown in Figure~\ref{fig:Z-uplift}.

The first thing we would need to do is re-scaling the adjacency matrix into the hypergraph tensor with suitable combinatorial factors (as in Definition~\ref{def:uplift-tensor}). However, here we can omit this step, as this factor is the same for all components. This implies that it will only modify the $\cZ$-eigenvalue, but that is something we will already compute.

The $\cZ$-eigenvector equation of $\wH$ decouples into three distinct ones:
\begin{itemize}
    \item Three equations for the centrality of the original nodes ($i\in\{1,2,3\}$),
    \begin{equation}
        \lambda c_i = \sum_{j,k,l,m} T_{ijklm} c_j \, c_k \, c_l \, c_m = \frac{4!}{2!} \sum_{j=1}^n T_{ij\star\times\times} c_j \, c_\star \, c_\times^2 = 12 c_\star \, c_\times^2\sum_{j=1}^n  A_{ij} c_j,
    \end{equation}
    where this combinatorial factor is the product of fixing 4 indices, out of which 2 are repeated.

    \item An equation for the centrality of the auxiliary node $\star$.
    \begin{equation}
        \lambda c_\star = \sum_{j,k,l,m} T_{\star jklm} c_j c_k c_l c_m 
        = \frac{4!}{2!} (T_{\star 12\times \times } c_1 + T_{\star 2 3 \times \times} c_3)\,c_2\,c_\times^2 =  12 (A_{12} c_1 + A_{23} c_3) \, c_2 \, c_\times^2.
    \end{equation}
    
    \item An equation for the centrality of the auxiliary node $\times$.
    \begin{equation}
        \lambda c_\times = \sum_{j,k,l,m} T_{\times jklm} c_j c_k c_l c_m = 4! (T_{\times 12\star \times } c_1 + T_{\times 2 3 \star \times} c_3)\,c_2\,c_\star\, c_\times = 24(A_{12} c_1 + A_{23} c_3) \, c_2 \, c_\star \, c_\times.
    \end{equation}
\end{itemize}

If rescale $\lambda' = \lambda/(12 c_\star c_\times^2)$, we have that the first of them becomes $\lambda' \bc = A \bc $; in other words, it is the eigenvector equation of the adjacency matrix of the original graph $G$. As it is connected, we are guaranteed to have a unique, positive solution $\bc>0$. 

The remaining equations are then (almost) completely fixed, as after re-scaling the eigenvalue leads to
\begin{equation}
    \lambda' c_\star^2 = (A_{12}c_1+A_{23}c_3)\,c_2, \quad 
    \lambda' c_\times^2 = 2 (A_{12}c_1+A_{23}c_3)\,c_2.
\end{equation}
which not only enforces $\sqrt{2}c_\star=c_\times$ but also guarantees their positivity, as $A_{12}=A_{23}$ and $\bc>0$. 

There is yet a subtlety to take into account: even though the Perron eigenvector $\bc$ can be rescaled as $\bc'=\alpha \bc$, the $\cZ$-eigenvector including $c_\star,c_\times$ cannot, it requires some normalization ($\bc^T\bc=1$ or $|\bc|_1=1$), which will force upon your solution the suitable value of $\alpha$.

With all these taken into account, we find the unique, positive solution $\tilde{c} = (\bc, c_\star, c_\times)^T$ to the problem to be 
\begin{equation}
    \cZ_1:  \, \tilde{c} = \frac{1}{4+2\sqrt{2}} \left(1, \sqrt{2}, 1, \sqrt{2}, 2 \right)^T; \quad
    \cZ_2:  \, \tilde{c} = \frac{1}{5} \left(1, \sqrt{2}, 1, \sqrt{2}, 2\right)^T.
\end{equation}

\end{example}

We can finally obtain the following sufficient condition for existence and uniqueness of certain $\cZ$-eigenvectors of tensors.
\begin{corollary}[Sufficient condition for the existence of the Perron-like $\cZ$-eigenvector]
    Let $T$ be a symmetric tensor of order $m>2$. If its associated hypergraph $H$ is strongly connected and can be seen as an uplift from a pairwise graph $G$, then a Perron-like $\cZ$-eigenvector of $T$ (i.e. a unique, positive $\cZ$-eigenvector) is guaranteed to exist.
\end{corollary}

\begin{proof}
    This follows directly from the Perron-Frobenius theorem, as it guarantees the existence and uniqueness of the eigenvector $\bc$ of the graph $G$ and the fact that the remaining (auxiliary nodes) equations fix uniquely (after choosing their positive values) these components in terms of $\bc$. 
\end{proof}

The only thing left for us to discuss is the connection between this multilinear algebra result, and our original perspective, which was that of hypergraph centralities. But making this leap is rather evident.
\begin{corollary}[Sufficient condition for the uniqueness of ZEC]
    If a hypergraph $H$ is strongly connected and can be seen as an uplift from a pairwise graph $G$, then it has a unique Perron-like $\cZ$-eigenvector.
\end{corollary}

It is important to remark that the computation of $\cZ$-eigenvalues is particularly complicated (See e.g. \cite{kolda2014adaptive,benson2019three,benson2019computing}), however our work clears the path for the simple computation of a whole class of hypergraphs.

%% file: conclusions.tex
In this study, we introduced a novel approach to analyze non-uniform hypergraphs by transforming them into an uniform hypergraph with the addition of an auxiliary node and suitably adjusting the weights of the transformed hyperedges, in an operation which we refer to as the ``uplift''. This transformation enabled us to apply well-defined centrality measures based on the eigenvectors of the resulting adjacency tensor. Through extensive comparisons with existing centrality measures in the literature, we have demonstrated the efficacy and relevance of our approach.

The key contribution of this work lies in the ability to bridge the gap between non-uniform hypergraphs and well-established centrality metrics. By introducing the auxiliary node, we effectively translated complex, multifaceted relationships into a format that aligns with already established hypergraph analysis techniques based on the $\cH$-eigenvector centrality, in a manner that is consistent due to the weighting choice. This, when supplemented with a projection operation, furnishes a sensible, well-defined novel centrality measure which, while at the same time retaining some degree of granularity (in the order which we can put the focus on), yields similar results, hence agreeing on the most important nodes of a hypergraph.

Our results showcased the advantages of our approach over existing methods: on the one hand the uniformization allows us to incorporate more information to the centrality when compared to uniform methods, on the other hand computing the adjacency tensor has a much lower computational complexity than the single other method available in the literature. Moreover, from an algebraic point of view we see that a generalization of the uplift to different nodes sheds light on the characterization of $\cZ$-eigenvectors of tensors, in particular it provides a simple route to their computation for a particular class of hypergraphs.

In summary, our study has presented a promising framework for the analysis of non-uniform hypergraphs, making them amenable to well-defined centrality measures based on tensor eigenvectors. This advancement holds great potential for applications across various domains, including social networks, biological systems, transportation networks, and beyond. By providing a bridge between complex, non-uniform relationships and established network analysis techniques, our approach contributes to a deeper understanding of the underlying structures and the identification of critical nodes within these intricate systems.